\documentclass[reqno]{amsart}

\usepackage{amssymb,amsbsy,amsthm,amsmath,graphicx,enumerate,xcolor}

\newtheorem{thm}{Theorem}[section]
\newtheorem{lemma}[thm]{Lemma}
\newtheorem{prop}[thm]{Proposition}
\newtheorem{cor}[thm]{Corollary}

\newtheorem*{thm*}{Theorem}

\theoremstyle{definition}
\newtheorem{definition}[thm]{Definition}

\newtheorem{question}[thm]{Question}
\theoremstyle{remark}
\newtheorem*{ex}{Example}
\newtheorem{remark}[thm]{Remark}

\def\ee{\mathrm{e}}

\def\diag{\mathrm{diag}}
\def\leb{\mathrm{m}}

 \newcommand{\one}{\mathbf{1}}
\def\supp{\mathrm{\mathop{supp}}}
\def\Ell{\mathrm{L}}

\newcommand{\T}{\mathbb{T}}

\def\lrLambda{\langle\Lambda\rangle}

\newcommand{\N}{\mathbb{N}}

\newcommand{\Z}{\mathbb{Z}}
\newcommand{\R}{\mathbb{R}}
\newcommand{\Q}{\mathbb{Q}}
\newcommand{\C}{\mathbb{C}}

\newcommand{\lin}{\mathrm{lin}}
\newcommand{\aveN}{\frac{1}{N}\sum_{n=1}^N}

\newcommand{\limaveN}{\lim_{N\to\infty} \aveN}
\newcommand{\aveNN}{\frac{1}{2N+1}\sum_{n=-N}^N}

\newcommand{\aveTau}{\frac{1}{\tau}\int_{0}^\tau}
\newcommand{\limaveTau}{\lim_{\tau\to\infty} \aveTau}

\renewcommand{\Hat}[1]{\widehat{#1}}
\newcommand{\ol}[1]{\overline{#1}}

\def\Fix{\mathop{\mathrm{Fix}}}
\def\Re{\mathop{\mathrm{Re}}}
\def\Im{\mathop{\mathrm{Im}}}

\def\atom{\text{ \upshape atom}}

\newenvironment{iiv}{\begin{enumerate}[{\upshape(i)}]}{\end{enumerate}}
\newenvironment{abc}{\begin{enumerate}[{\upshape(a)}]}{\end{enumerate}}

\def\dd{\mathrm{d}}
\def\LLL{\mathcal{L}}
\def\sprod#1#2{(#1|#2)}
\def\dprod#1#2{\langle #1,#2\rangle}
\def\cnu{\mathrm{cnu}}
\def\abs#1{|#1|}
\def\uni{\mathrm{uni}}

\def\Dlim{\mathop{\mathrm{D\text{-}}\lim}}
\allowdisplaybreaks

\def\cconv{\overline{\mathop{\mathrm{conv}}}}
\def\Id{I}
\begin{document}

\title[Wiener's lemma along  subsequences]{Wiener's lemma along primes and other subsequences}

\author{Christophe Cuny}
\author{Tanja Eisner}
\author{B\'alint Farkas}

\address{Institut de Sciences Exactes et Appliqu\'ees, University of New-Caledonia, 
New-Caledonia}
\email{christophe.cuny@univ-brest.fr}

\address{Institute of Mathematics, University of Leipzig\newline
P.O. Box 100 920, 04009 Leipzig, Germany}
\email{eisner@math.uni-leipzig.de}
\address{School of Mathematics and Natural Sciences, University of Wuppertal\newline Gau\ss stra\ss e 20,
42119 Wuppertal, Germany}
\email{farkas@math.uni-wuppertal.de}


\thanks{B.F. gratefully acknowledges of the support of the DAAD-Tempus PPP Grant  $\#$ 57448965}

\maketitle

\begin{abstract}
Inspired by subsequential ergodic theorems, we study the validity of
Wiener's lemma and the extremal behavior of a measure $\mu$ on the unit circle via the behavior of its Fourier coefficients $\Hat\mu(k_n)$ along subsequences $(k_n)$. We focus on arithmetic
subsequences such as polynomials, primes and polynomials of primes, and also discuss connections to rigidity sequences, return times sequences and strongly sweeping out sequences as well as 
measures on $\R$.   We also present consequences for orbits of operators and of $C_0$-semigroups on Hilbert and Banach spaces extending the results of Goldstein \cite{G2} and Goldstein, Nagy \cite{G3}.  The results are complemented by some open questions and indication of interesting research directions.

\medskip
{\color{red} After this paper had been published, it  was pointed out to us by 
Emmanuel Lesigne and M\'at\'e Wierdl that there is a gap in the Example on return 
times sequences along polynomials on page \pageref{page:polret}. Indeed, to make the 
argument there work, one needs a Wiener--Wintner type result for polynomial 
averages with a precise information about the limit, and this is presently out of reach.}
\end{abstract}

\section{Introduction}
Wiener's lemma is a classical result connecting the asymptotic behavior of the Fourier coefficients
\[
\Hat\mu(n)=\int_\T z^n\dd\mu(z)
\]
of a complex Borel measure $\mu$ on the unit circle $\T$ with its values on singletons. Despite its elementary proof, it has found remarkable applications in several areas of mathematics such as ergodic theory, operator theory,  group theory and number theory.

\begin{thm}[Wiener's Lemma]\label{thm:Wiener}
Let $\mu$ be a complex Borel measure\footnote{By definition finite.} on the unit circle $\T$. Then
\[
\limaveN |\Hat{\mu}(n)|^2=\sum_{a \text{ \upshape atom}} |\mu(\{a\})|^2.
\]
\end{thm}
\noindent (Since $\ol{\Hat{\mu}}(-n)={\Hat{\ol{\mu}}(n)}$, one can replace here $\aveN $ by $\aveNN$.)

As a consequence, one has the following characterization of Dirac measures in terms of their Fourier coefficients. Here we restrict ourselves to probability measures and give the proof for the reader's convenience.
\begin{cor}[Extremal behavior of Dirac measures]\label{cor:W-extremal}
For a Borel probability measure $\mu$ on $\T$ the following assertions are equivalent:
\begin{iiv}
\item $\lim\limits_{N\to\infty}\aveN |\Hat{\mu}(n)|^2=1$.
\item $\lim\limits_{n\to\infty} |\Hat{\mu}(n)|=1$.
\item $\mu$ is a Dirac measure.
\end{iiv}
\end{cor}
Note that by the Koopman--von Neumann lemma, see Lemma \ref{lem:KvN} (b) below and the paragraph preceding it, (i) is equivalent to $|\Hat{\mu}(n)|\to 1$ in density as $n\to\infty$ i.e., to the existence of a subset $J$ of density $1$ with $\lim_{n\to\infty,n\in J}|\Hat{\mu}(n)| =1$.
\begin{proof}
Let $\mu$  be a probability measure on $\T$.
If $\mu(\{a\})=1$ for some $a\in\T$, then $\Hat{\mu}(n)=a^n$. Whence the implication (iii)$\Rightarrow$(ii) follows, while (ii)$\Rightarrow$(i) is trivial. To show (i)$\Rightarrow$(iii), suppose (i). Theorem \ref{thm:Wiener} 
yields
\[
1=\limaveN |\Hat{\mu}(n)|^2=\sum_{a \atom} \mu(\{a\})^2\leq
\sum_{a \atom} \mu(\{a\}) =1,
\]
implying $\mu(\{a\})^2=\mu(\{a\})$  or, equivalently, $\mu(\{a\})\in \{0,1\}$ for every atom $a$. We conclude that $\mu$ is a Dirac measure.
\end{proof}
The previous two results have the following operator theoretic counterparts.
\begin{thm}[Goldstein \cite{Goldstein1}, Ballotti, Goldstein \cite{G1}]\label{abstwiener}
Let $T$ be a (linear) contraction on a Hilbert space $H$, and for $\lambda\in \C$ denote by $P_\lambda$ the orthogonal projection onto $\ker(\lambda-T)$. Then for every $x,\:y\in H$
\begin{align*}
\lim_{N\to\infty}\frac1N\sum_{n=1}^N\abs{\sprod{T^nx}{y}}^2&=\sum_{\lambda\in\T}  \abs{\sprod {P_\lambda x}{P_\lambda y}}^2=\sum_{\lambda\in\T}  \abs{\sprod {P_\lambda x}{y}}^2.
\end{align*}
\end{thm}
It is easy to deduce the previous theorem from Wiener's lemma (and vice versa), even though the original proof of Goldstein went along different lines. 

The following Banach space version of Corollary \ref{cor:W-extremal} is more complex, see also Lin \cite{L} and Baillon,  Guerre-Delabri\`ere \cite{BG} for related results. Note that in these papers, the results are formulated for $C_0$-semigroups but are also valid for powers of operators with analogous proofs.
\begin{thm}[Goldstein, Nagy \cite{G3}]\label{abstwienerextr}
Let $T$ be a (linear) contraction on a Banach space $E$. Suppose for some $x\in E$
\begin{align*}
|\dprod{T^nx}{x'}|\to |\dprod{x}{x'}|\quad \text{for every $x'\in E'$ as $n\to \infty$.}
\end{align*}
Then $(\lambda-T)x=0$ for some $\lambda\in \T$.
\end{thm}

\smallskip

 The aim of this paper is to systematically study the validity of Wiener's Lemma along subsequences and the corresponding consequences in the spirit of the above.  The following is an example of our results, see Theorem \ref{thm:wiener-pol-prim} below.
\begin{thm*}[Wiener's lemma along polynomials and primes]
Let $p_n$ be the $n^\text{th}$ prime, let $P$ be a non-constant
integer polynomial,  and for $n\in\N$ let $k_n:=P(n)$ or $k_n:=P(p_n)$. Then for 
every finite, positive Borel measure $\mu$ on $\T$ 
\begin{align*}
\limaveN |\Hat{\mu}(k_n)|^2&=\sum_{a/q\in\Q} c(e(a/q)) \sum_{\lambda
\text{\atom}} \mu(\{\lambda\}) \mu(\{\lambda e(a/q)\})\\&
\notag
\leq
\sum_{\lambda\in U} \mu(\{\lambda e(\Q)\})^2,
\end{align*}
where $U$ is a maximal set of rationally independent atoms of $\mu$, $e(x)=\ee^{2\pi i x}$ for $x\in \R$ and for $\lambda\in \T$
\[
c(\lambda)=\limaveN \lambda^{k_n}
\]
(the limit exists, see \cite{RW} and below).
\end{thm*}
As a consequence we immediately obtain that for a continuous measure $\mu$ one has   $|\Hat\mu(k_n)|\to 0$ as $n\to \infty$ in density, which is non-trivial in view of the fact that the sequences $(k_n)$  as above have density $0$ (except in the trivial case of linear $P$).

In our treatment of the generalizations of Wiener's Lemma,  Corollary \ref{cor:W-extremal} and Theorems \ref{abstwiener}, \ref{abstwienerextr} we will separate the  study of the equivalences (i)$\Leftrightarrow$(iii) and (ii)$\Leftrightarrow$(iii) in Corollary \ref{cor:W-extremal}. 

\medskip\noindent First of all, here are some words about terminology. The term \emph{complex measure} refers to $\C$-valued $\sigma$-additive set function (which is then  automatically finite valued, and has finite variation). In this paper only Borel measures will be considered.
A \emph{subsequence} $(k_n)$ in $\N$ will refer to a function
$k:\N\to \N$ which is strictly increasing for sufficiently large
indices. Banach and Hilbert spaces will be considered over the complex field $\C$.

Sequences for which Wiener's lemma and the extremality of Dirac measures work well include certain polynomial sequences, the primes, certain polynomials of primes, as has been said above, and certain return times sequences as will be shown below. 

As an application of the general results we shall prove among others
the following, maybe at first glance surprisingly looking, facts. Recall that $p_n$ denotes the $n^\text{th}$ prime.
\begin{enumerate}[1)]
\item The only  Borel probability measures on $\T$ with $|\Hat\mu(p_n)|\to 1$ for $n\to \infty$ are the Dirac measures (Theorem \ref{thm:pol-primes-extr}).
\item If $T$ is a (linear) contraction on a Hilbert space and $x\in H\setminus\{0\}$ is such that $|\sprod{T^{p_n}x}{x}|\to \|x\|^2$ as $n\to \infty$, then $x$ is an eigenvector of $T$ to a unimodular eigenvalue (Theorems  \ref{thm:pol-primes-extr} and \ref{thm:GNgeneral}).
\item If $T$ is a power bounded operator on a Banach space $E$ and $x\in E\setminus\{0\}$ is such that $|\dprod{ T^{p_n}x}{x'}|\to |\dprod{x}{x'}|$ as $n\to \infty$ for every $x'\in E'$, then $x$ is an eigenvector of $T$ to a unimodular eigenvalue (Corollary \ref{cor:primesgelfand}).
\end{enumerate}
We also relate our results to \emph{rigidity sequences} and discover a 
property of such sequences as a byproduct which appears to be new. As a consequence, we show that for a non-constant integer polynomial $P$ the sequence $(P(p_n))$ is not a rigidity sequence. We also discuss connection to \emph{strongly sweeping out sequences}, see the end of Section \ref{sec:extr}.

Our results are inspired by ergodic theory, where the study of ergodic theorems along subsequences has been a rich area of research with connections to harmonic analysis and number theory. Furstenberg \cite{F-book} 
described norm convergence of ergodic averages of unitary operators along polynomials.
Pointwise convergence of ergodic averages for measure preserving transformations along polynomials and primes, answering a question of Bellow and Furstenberg, was proved by Bourgain \cite{B86,B88,B} and Wierdl \cite{W}, with polynomials of primes treated by Wierdl \cite{W-diss} and Nair \cite{N,N1}.
To illustrate the wealth of literature on  ergodic theorems along subsequences  we refer, e.g., to  Bellow \cite{Be}, Bellow, Losert \cite{BeL},  Baxter, Olsen \cite{BO}, Rosenblatt, Wierdl \cite{RW}, 
Akcoglu, Bellow, Jones, Losert, Reinhold-Larsson, Wierdl \cite{Akcoglu-etal},
 Berend, Lin, Rosenblatt, Tempelman \cite{BLRT}, Boshernitzan, Kolesnik, Quas, Wierdl \cite{BKQW}, Krause \cite{K}, Zorin-Kranich \cite{ZK-primes}, Mirek \cite{M}, Eisner \cite{E}, Frantzikinakis, Host, Kra \cite{FHK}, Wooley, Ziegler \cite{WZ}.

The paper is organized as follows. Section \ref{sec:Wiener-subseq} is devoted  to an abstract version of Wiener's lemma along subsequences. In Section \ref{sec:extr}  we study extremal and Wiener extremal subsequences,  see Definition \ref{def:extr}. The case of polynomials, primes and polynomials of primes is treated in Section \ref{sec:pol-primes}.  Section \ref{sec:operators} is devoted to applications to orbits of operators on Hilbert and Banach spaces.
The continuous parameter case is discussed in Section \ref{sec:cont}, where parallels and differences to the time discrete case are pointed out.

\textbf{Acknowledgment.} The authors thank Michael Lin, Rainer Nagel and J\'anos Pintz for the helpful comments and references. We are also greatly indebted to the anonymous referees for making suggestions to improve the presentation of the paper and for bringing the strong sweeping out property to our attention, thus indicating a new research direction.

%%%%%%%%%%%%%%%%%
%                                                %
% Wiener along subsequences  %
%                                                %
%%%%%%%%%%%%%%%%%

\section{Wiener's Lemma  along subsequences}\label{sec:Wiener-subseq}

Recall that a sequence $(a_n)$ in $\C$ is called \emph{convergent in density} to $a\in\C$, with notation $\Dlim_{n\to\infty}a_n=a$ if there exists a set $J\subseteq \N$ of density $1$ with $\lim_{n\to\infty,n\in J}a_n=a$. The density of a set $J\subseteq\N$ is defined by $\lim_{n\to\infty}\frac{|J\cap\{1,\ldots,n\}|}{n}$, provided the limit exists.

The following is the classical Koopman--von Neumann lemma together with a slight variation.
\begin{lemma}\label{lem:KvN}
\begin{abc}
\item For a bounded sequence $(a_n)$ in $[0,\infty)$ the following are equivalent:
\begin{iiv}
\item $\Dlim_{n\to\infty}a_n=0$.
\item $\limaveN a_n=0$.
                  \item $\limaveN a_n^2=0$.
\end{iiv}
\item For a bounded sequence $(b_n)$ in $(-\infty,1]$ the following are equivalent:
\begin{iiv}
\item $\Dlim_{n\to\infty}b_n= 1$.
\item $\limaveN b_n= 1$.
\end{iiv}
If $b_n\geq 0$, then these assertions are also equivalent to:
\begin{iiv}\setcounter{enumii}{2}
\item $\limaveN b_n^2= 1$.
\end{iiv}
\end{abc}
\end{lemma}
\begin{proof}(a), (i)$\Leftrightarrow$(ii) is the content of the Koopman--von Neumann lemma, see, e.g., \cite{Koopman-vonNeumann1932} or \cite[Ch.{} 9]{EFHN}, whereas (i)$\Leftrightarrow$(iii) is a direct consequence. 
(b) follows from (a) by considering $a_n:=1-b_n$.
\end{proof}

We further recall the following notion from Rosenblatt, Wierdl \cite{RW}, see also \cite[Chapter{} 21]{EFHN}.
\begin{definition}
A subsequence $(k_n)$ of $\N$ is called \emph{good} if
\[
 \text{ for every }\lambda\in\T\text{ the limit}\quad \limaveN \lambda^{k_n}=: c(\lambda) \quad\text{exists}.
\]
Moreover,  $(k_n)$ is called an \emph{ergodic sequence} if $c=\one_{\{1\}}$, the characteristic function of $\{1\}$.
We call the set $\Lambda:=\{\lambda: c(\lambda)\neq 0\}$ the \emph{spectrum of the sequence} $(k_n)$ in analogy to, e.g., Lin, Olsen, Tempelman \cite{LOT}.
\end{definition}

\begin{remark}
Given a sequence $(k_n)$ which is not necessarily good, we shall still denote by $c(\lambda)$ the limit  $\limaveN \lambda^{k_n}$ whenever it exists. 
\end{remark}

By an application of the spectral theorem it follows that a sequence $(k_n)$ is good if and only if it is \emph{good for the mean ergodic theorem}, that is if for every measure preserving system $(X,\nu,T)$ and every $f\in \Ell^2(X,\nu)$ the averages
\[
\aveN T^{k_n}f
\]
converge in $\Ell^2(X,\nu)$, where $T$ denotes the \emph{Koopman operator} corresponding to the transformation $T$ defined by $f\mapsto f\circ T$. A good sequence is then ergodic if and only if the above limit always equals the orthogonal projection $P_{\Fix (T)}f$ onto the fixed space $\Fix (T)=\ker(1-T)$.

\begin{remark}\label{rem:c}
\begin{abc}
\item For each subsequence $(k_n)$ of $\N$, the
sequence $(\lambda^{k_n})$ is equidistributed in $\T$ for almost every
$\lambda\in\T$ implying that $c(\lambda)$ exists and is $0$ for Lebesgue almost every $\lambda\in \T$, see,
e.g., Kuipers, Niederreiter \cite[Theorem 1.4.1]{KN} (or Theorem 2.2 on
page 50 of \cite{RW}). 
The function $c:\T\to \C$ is Borel measurable (if it exists).
\item The function $c$ clearly satisfies $c(1)=1$,
$c(\ol{\lambda})=\ol{c(\lambda)}$ and $|c(\lambda)|\leq1$ whenever $c(\lambda)$ exists. 
Moreover, if $|c(\lambda)|=1$, then $\lambda^{k_n}$ converges to $c(\lambda)$ in density, which follows  by Lemma \ref{lem:KvN} (b) applied to $a_n:=\Re (\ol{c(\lambda)}\lambda^{k_n})$. Thus $c$ is a multiplicative function on the subgroup $\{\lambda:\, |c(\lambda)|=1\}$ of $\T$. The fact that 
$\{\lambda:\, |c(\lambda)|=1\}$ is a subgroup  of $\T$ is actually a particular case of a more general result of Akcoglu, Bellow, Jones, Losert,  Reinhold-Larsson and Wierdl, see  \cite[Lemma 3.12]{Akcoglu-etal}. 
\end{abc}
\end{remark}

We present one more property of the limit function $c$. For 
%an integer 
$d\in\N$ we set $G_d:=\{\lambda\in \T\, :\, \lambda^d=1\}$, the group of $d^\text{th}$ roots of unity.

\begin{prop}\label{prop:c=1}
Let $(k_n)$ be a good sequence with corresponding limit function $c$. Then there exists %an integer 
$d\in \N$ such that 
$\Gamma:=\{\lambda\in \T\, :\, |c(\lambda)|=1\}=G_d$.
\end{prop}

\begin{proof} 
 It follows from Remark \ref{rem:c} that $\Gamma$ 
is a group, and it is then  well-known that  $\Gamma$ is either finite or dense in $\T$. We shall prove that it is finite. 

\medskip\noindent Since $(k_n)$ is good, $c$ is the pointwise limit of a sequence of continuous 
functions on a compact space. Hence, by a theorem of Baire, its set of continuity points is dense in $\T$. 
As mentioned in Remark \ref{rem:c},  
 $\lim_{N\to \infty}\frac1N\sum_{n=1}^N \lambda ^{k_n}=
0$ for almost every $\lambda$ (with respect to the Haar measure on $\T$).
If $\Gamma$ is not finite, we infer that $c$ is nowhere continuous, which is impossible.
\end{proof}

The following
general fact may appear to be well known, but we could not find a reference.

\begin{prop}[Wiener's lemma along subsequences]\label{prop:W-seq}
Let $(k_n)$ be a good sequence in $\N$.
\begin{abc}
\item For every complex Borel measure
$\mu$ on $\T$
\[
\limaveN |\Hat{\mu}(k_n)|^2=\int_{\T^2} c(\lambda_1 \ol{\lambda_2})
\dd(\mu\times\overline{\mu})(\lambda_1,\lambda_2).
\]
\item The sequence $(k_n)$ is ergodic if and only if
\[
\limaveN |\Hat{\mu}(k_n)|^2=\sum_{a
\atom} |\mu(\{a\})|^2
\]
holds for every complex Borel measure $\mu$ on $\T$.
\item For an ergodic sequence $(k_n)$ and a Borel probability measure $\mu$ on $\T$ the limit above in {\upshape(b)} is $1$ if and only if $\mu$ is a Dirac measure.
\end{abc}
\end{prop}

\begin{proof}
(a)
The proof goes along the same lines as the most elementary and well-known proof of the Wiener lemma. Observe that,
by Fubini's theorem and by Lebesgue's dominated convergence theorem,
\begin{align*}
\aveN  |\Hat{\mu}(k_n)|^2&=
\aveN \int_\T \lambda_1^{k_n} \dd\mu(\lambda_1)\int_\T
\ol{\lambda_2}^{k_n} \dd\ol{\mu}(\lambda_2)\\
&=\int_{\T\times\T} \aveN (\lambda_1\ol{\lambda_2})^{k_n}
\dd(\mu\times\ol{\mu})(\lambda_1,\lambda_2)\\
&\to \int_{\T^2} c(\lambda_1 \ol{\lambda_2})
\dd(\mu\times\ol{\mu})\quad \text{  as } N\to\infty.
\end{align*}

\medskip\noindent (b)
If now $c=\one_{\{1\}}$ we see that, by Fubini' theorem, the limit above equals
\[
\int_\T \left(\int_\T
\one_{\{\lambda_1\}}(\lambda_2)\dd\mu(\lambda_1)\right)\dd\ol{\mu}(\lambda_2)
=\int_\T \mu(\{\lambda_2\})\dd\ol{\mu}(\lambda_2)=\sum_{a \atom} |\mu(\{a\})|^2.
\]
For the converse implication suppose $(k_n)$ is not ergodic, and let $\lambda\in \T\setminus\{1\}$ be with $c(\lambda)\neq 0$. If $\Re c(\lambda)\neq 0$, then consider the probability measure $\mu:=\frac12(\delta_1+\delta_\lambda)$. We then have
\[
\int_{\T^2} c(\lambda_1 \ol{\lambda_2})
\dd(\mu\times\overline{\mu})(\lambda_1,\lambda_2)=\frac12+\frac14\Bigl(c(\lambda)+c(\ol{\lambda})\Bigr)\neq \frac12=\sum_{a
\atom} |\mu(\{a\})|^2.
\]
If $\Im c(\lambda)\neq 0$, then for the measure $\mu:=\frac12(\delta_1+i\delta_\lambda)$ we have \[
\int_{\T^2} c(\lambda_1 \ol{\lambda_2})
\dd(\mu\times\overline{\mu})(\lambda_1,\lambda_2)=\frac12+\frac i4\Bigl(c(\lambda)-c(\ol{\lambda})\Bigr)\neq \frac12=\sum_{a
\atom} |\mu(\{a\})|^2.\] The proof of (b) is complete.

\medskip\noindent (c) follows from (b) by a similar arguments as in the proof of Corollary \ref{cor:W-extremal}.
\end{proof}

The following questions arise naturally, cf.{} also Proposition \ref{prop:rigid} below.
\begin{question}
Does the existence of $\limaveN |\Hat{\mu}(k_n)|^2$ for every  probability Borel measure $\mu$ imply that $(k_n)$ is good?
Is there a non-ergodic, good sequence $(k_n)$ with $\Re c(\lambda)=0$ for each $\lambda\in \T\setminus\{1\}$?
\end{question}

\begin{remark} If one replaces ``probability measure''  by ``complex measure'', the answer to the first question is positive. Indeed, it is easy to see that for a subsequence  $(k_n)$ of $\N$ the following assertions are equivalent.
\begin{iiv}
\item  For every finite complex  measure (respectively every probability measure)
$\mu$ on $\T$ the limit $\lim_{N\to \infty} \frac1N \sum_{n=1}^N|\hat \mu (k_n)|^2$ exists;
\item  For every complex measure (respectively every 
probability measure)
$\mu$ and $\nu$ on $\T$ the limit $\lim_{N\to \infty} \frac1N 
\sum_{n=1}^N \Re (\hat \mu (k_n)\overline{\hat \nu (k_n)})$ exists.
\end{iiv}
Assume that the above equivalent conditions hold for probabilities. 
Taking $\nu:=\delta_1$ and $\mu:=\delta_\lambda$, we see that 
the limit $\lim_{N\to \infty} \frac1N 
\sum_{n=1}^N \Re (\lambda ^{k_n})$ exists.
If we assume moreover that the conditions hold for finite complex 
measures, then, taking $\nu:=i\delta_1$ and $\mu:=\delta_\lambda$, 
we see that 
the limit $\lim_{N\to \infty} \frac1N 
\sum_{n=1}^N \Im (\lambda ^{k_n})$ exists,  implying that $(k_n)$ is good. 
\end{remark}

\begin{cor}\label{cor:coset}
Let $(k_n)$ be a good sequence. 
For a Borel probability measure $\mu$
\begin{equation}\label{eq:extr-seq}
\limsup_{N\to \infty}\frac1N\sum_{n=1}^N |\hat \mu(k_n)|^2
=1
\end{equation}
holds if and only if $\mu$ is discrete with
\begin{equation}\label{eq:atoms}
c(a\ol{b})=1\quad  \text{for all atoms }a,b.
\end{equation}
In this case, the limit superior is a limit, and $\mu$ is supported in a coset of  $G_d$ for some $d\in \N$.
\end{cor}
\begin{proof} Suppose \eqref{eq:extr-seq} holds. 
Since $(k_n)$ is good, by  Proposition \ref{prop:W-seq} (a)  the above limit superior is actually a limit and 
\[
\frac1N\sum_{n=1}^N |\hat \mu(k_n)|^2 \underset{N\to\infty}\longrightarrow \int_{\T\times \T} 
c(\lambda_1\bar \lambda_2) \dd\mu(\lambda_1)\dd\mu(\lambda_2)=1.
\]
Hence
\[
\int_{\T\times \T} \big(1-\Re(c(\lambda_1\bar \lambda_2))\big) \dd\mu(\lambda_1)\dd\mu(\lambda_2)=0 .
\]
Since $1- \Re(c(\lambda_1\bar \lambda_2))\geq 0$ (and $|c|\le 1$). We infer that 
there exists $\lambda_2\in \T$ such that for $\mu$-a.e. $\lambda_1\in \T$, 
$c(\lambda_1\bar \lambda_2)=1$. Hence, $\mu$ is supported on 
$\lambda_2\Gamma$, which equals $\lambda_2 G_d$ for some integer $d\geq0$ by Proposition \ref{prop:c=1}. This shows one implication. 

\medskip\noindent For the converse implication let $\mu$ be discrete satisfying \eqref{eq:atoms}. Then by Proposition \ref{prop:W-seq} (a) 
\[
\frac1N\sum_{n=1}^N |\hat \mu(k_n)|^2 =\sum_{a,b\atom} c(a\ol{b})\mu(\{a\})\mu(\{b\})=\sum_{a,b\atom}\mu(\{a\})\mu(\{b\})=1.
\]
 \end{proof}

We now consider the case ``in between'', namely when $c(\lambda)=0$
for all  but at most countably many $\lambda$'s. We first introduce
the following terminology: For a subset $\Lambda$ of $\T$ denote by $\lrLambda$ the subgroup generated $\Lambda$. We call two elements 
$\lambda_1,\lambda_2\in\T$  \emph{$\Lambda$-dependent} if their cosets with respect to $\lrLambda$ coincide: $\lambda_1\lrLambda=\lambda_2\lrLambda$, otherwise we call them \emph{$\Lambda$-independent}.

\begin{thm}\label{thm:count}
Let $(k_n)$ be a good sequence with at most countable spectrum $\Lambda$.
\begin{abc}
\item For every complex Borel measure $\mu$ on $\T$
\[
\limaveN |\Hat{\mu}(k_n)|^2= \sum_{\lambda\in\Lambda}
c(\lambda)\sum_{a\atom}\ol{\mu}(\{a\ol{\lambda}\}) \mu(\{a\}).
\]
In particular,  for every continuous complex Borel measure $\mu$ on $\T$
\[
\limaveN |\Hat{\mu}(k_n)|^2= 0.
\]
\item
For every  Borel probability measure $\mu$ on $\T$
\begin{equation}\label{eq:wiener-subgroup}
\limaveN |\Hat{\mu}(k_n)|^2\leq \sum_{a\in U} \mu(a\lrLambda)^2,
\end{equation}
where $U$ is a maximal set of $\Lambda$-independent atoms. The equality in \eqref{eq:wiener-subgroup} holds if and only if
$\mu$ satisfies \eqref{eq:atoms} (but it may not necessarily be discrete).
\end{abc}
\end{thm}

\begin{proof}
(a)
By Proposition \ref{prop:W-seq} and Fubini's theorem we have
\begin{align*}
\limaveN |\Hat{\mu}(k_n)|^2&=\int_{\T^2} c(\lambda_1 \ol{\lambda_2})
\dd(\mu\times\ol{\mu})(\lambda_1,\lambda_2)\\
&=\int_\T \sum_{\lambda\in\Lambda} c(\lambda) \ol{\mu}(\{\lambda_1
\ol{\lambda}\}) \dd\mu(\lambda_1)\\
&=\sum_{\lambda\in\Lambda} c(\lambda)\int_\T  \ol{\mu}(\{\lambda_1
\ol{\lambda}\}) \dd\mu(\lambda_1)\\
&= \sum_{\lambda\in\Lambda}  c(\lambda) \sum_{a \atom}\ol{\mu}(\{a\ol{\lambda}\}) \mu(\{a\}).
\end{align*}

\medskip\noindent
(b) Observe (the left-hand side  below is greater or equal to
zero by (a))
\begin{align*}
\sum_{\lambda\in\Lambda}  c(\lambda) \sum_{a \atom} \mu(\{a\})
\mu(\{a \ol{\lambda}\})
&\leq \sum_{a \atom} \mu(\{a\})
\left|\sum_{\lambda\in\Lambda} c(\lambda)  \mu(\{a
\ol{\lambda}\})\right|\\
&\leq \sum_{a \atom} \mu(\{a\}) \sum_{\lambda\in\Lambda}
\mu(\{a \ol{\lambda}\})\\
&=\sum_{a \atom} \mu(\{a\}) \mu(a\lrLambda) \\
&= \sum_{a\in U} \mu(a\lrLambda)^2.
\end{align*}
The last assertion regarding the equality is clear.
\end{proof}

We will see below that there are sequences $(k_n)$ satisfying the
assumptions of Theorem \ref{thm:count} and probability measures
satisfying \eqref{eq:extr-seq} which are  not Dirac.

\begin{remark}
Let $(k_n)$ be a strictly increasing sequence in $\N$ having positive density. 
If the characteristic function of $\{k_n:n\in \N\}$ is a Hartmann sequence (i.e., has Fourier coefficients), then $(k_n)$ is good with at most countable spectrum (see, e.g., Lin, Olsen, Tempelman \cite{LOT} or Kahane \cite{Kah}).
\end{remark}

By using a result\footnote{We thank Michael Lin for bringing the reference   \cite{Rosenblatt} to our attention.} of Boshernitzan, 
which we recall
 for the sake of completeness, it is possible to show that good sequences with positive upper density have countable spectrum.
Actually,   his Theorem 41  in Rosenblatt \cite{Rosenblatt} is stated in the case where {$k_n=n$ for every $n\in \N$},  but the proof is the same. 
\begin{prop}[Boshernitzan, see Rosenblatt \cite{Rosenblatt}]\label{prop-BR}
Let $(a_n)$ be a bounded sequence of complex numbers and let $(N_m)$ be a subsequence of $\N$. For every $\delta>0$, 
the set 
\[
\Bigl\{\lambda \in \T\, :\, \liminf_{k\to \infty}\frac1{N_m} 
 \Bigl|\sum_{\ell=1}^{N_m}a_\ell \lambda ^\ell\Bigl|\geq \delta\Bigr\}\quad\text{is finite.}
 \]
\end{prop}

Recall that the \emph{upper density} of a  subsequence $(k_n)$ of $\N$ is defined by
 \[\ol{d}(k_n):=\limsup_{N\to\infty}\frac{|\{n:\, k_n
\leq N\}|}{N}.\]

\begin{cor}
Let  $(k_n)$  be a good subsequence of $\N$ with positive upper density. For every $\delta>0$ the set 
$\{\lambda \in \T\, :\, |c(\lambda)|\geq \delta\}$ is finite. In 
particular, $(k_n)$ has countable spectrum. 
\end{cor}
\begin{proof}
 By assumption, there exists a subsequence $(N_m)_{m\in \N}$ of $\N$ such that 
 \[
 \lim_{m\to\infty}\frac{|\{n:\, k_n
\leq N_m\}|}{N_m}= \gamma>0.
\]
Then, for $A=\{k_n:n\in \N\}$ and for every $\lambda\in \T$ we obtain
\[
\frac1{N_m}\sum_{\ell=1}^{N_m}\lambda^\ell\mathbf{1}_A(\ell)
\underset{m\to \infty}\longrightarrow \gamma c(\lambda).
\]
An application of Proposition \ref{prop-BR} with $a_\ell=\mathbf{1}_A(\ell)$ finishes the proof.
\end{proof}

\begin{remark}\label{rem:norel1}
A good sequence need not have positive upper density as, e.g., $k_n=n^2$ shows. See Section \ref{sec:pol-primes} below for this and other examples. On the other hand, a sequence with positive upper density (and even density) does not have to be good. Indeed, take $2\N$ and change $2n$ to $2n+1$ if $2n$ lies in any  interval of the form $[4^\ell,2\cdot4^{\ell}]$, $\ell\in \N$. This sequence has density $1/2$ but $c(-1)$ does not exist. Modifying this construction it is easy to construct a sequence with density arbitrarily close to $1$ which is not good. (Note that $1$ cannot be achieved: every sequence with density $1$ is automatically good.)
\end{remark}

\begin{remark}
Suppose $(k_n)$ is a subsequence of $\N$ (not necessarily good) such that there is an at most 
countable set $\Lambda$ such that $c(\lambda)$ exists and equals
$0$ for every $\lambda \not\in \Lambda$.
By carrying out the same calculation as in the proof of (a) in 
Proposition \ref{prop:W-seq} and using the Koopman--von Neumann Lemma \ref{lem:KvN}  we see that for each continuous measure on $\T$ we have $\Hat
\mu(k_n)\to 0$ in density. It would be interesting to characterize
those subsequences $(k_n)$ for which a (probability) measure is
continuous if and only if $\Hat \mu(k_n)\to 0$ in density.
\end{remark}

%%%%%%%%%%%%%%%
%                                           %
% Extremal  subsequences  %
%                                           %
%%%%%%%%%%%%%%%

\section{(Wiener) extremal subsequences}\label{sec:extr}

In this section we characterize subsequences $(k_n)$ for which the equivalences (i)$\Leftrightarrow$(iii) and (ii)$\Leftrightarrow$(iii) in Corollary \ref{cor:W-extremal} remain valid and show that (i)$\Leftrightarrow$(ii) fails in general.

\begin{definition}\label{def:extr}
Let $(k_n)$ be a subsequence of $\N$.
 We call a Borel probability measure $\mu$ \emph{Wiener extremal} or \emph{extremal} along $(k_n)$ if $\mu$ satisfies
\[
\limaveN |\Hat{\mu}(k_n)|^2=1 \quad\text{or}\quad \lim_{n\to\infty} |\Hat{\mu}(k_n)|=1,\quad\text{respectively}.
\] 
A subsequence $(k_n)$ in $\N$  is called \emph{(Wiener) extremal} if every (Wiener) extremal measure is a Dirac measure. If every (Wiener) extremal discrete measure is Dirac, then we call $(k_n)$ \emph{(Wiener) extremal for discrete measures}.
\end{definition}

We first consider  Wiener extremal sequences.
\begin{thm}\label{thm:extr-W}
For a subsequence $(k_n)$  of $\N$ consider the following assertions:
\begin{iiv}
\item $(k_n)$ is Wiener extremal.
\item $(k_n)$ is Wiener extremal for discrete measures.
\item For each $z\in \T$ whenever
\[
\Dlim_{n\to\infty} z^{k_n}\to 1,
\]
then $z=1$.
\item $c(\lambda)=1$ implies $\lambda=1$.
\end{iiv}
Then {\upshape (i)}$\Rightarrow${\upshape (ii)}$\Leftrightarrow${\upshape (iii)}$\Leftrightarrow${\upshape (iv)}. Moreover, {\upshape (i)}$\Leftrightarrow${\upshape (ii)} if $(k_n)$ is good.
\end{thm}
\begin{proof}(i)$\Rightarrow$(ii) is trivial and (iii)$\Leftrightarrow$(iv) follows from Remark \ref{rem:c}.

\medskip\noindent (iii)$\Rightarrow$(ii): Assume that there exists a discrete probability measure which is extremal and not Dirac. Let $a, b$ be two different atoms of $\mu$. Since
\[
|\Hat\mu(n)|\leq |a^{n}\mu(\{a\})+b^{n}\mu(\{b\})|+\sum_{\lambda\neq a,b \text{\ atom}}\mu(\{\lambda\})\leq 1,
\]
the extremality of $\mu$ implies that $|a^{k_n}\mu(\{a\})+b^{k_n}\mu(\{b\})|$ converges in density to $\mu(\{a\})+\mu(\{b\})$ or, equivalently, that  $|a^{k_n}-b^{k_n}|$ converges in density to $1$. Taking $z:=a\ol{b}\neq 1$ in (iii), we arrive at a contradiction.

\medskip\noindent(ii)$\Rightarrow$(iii): Assume that there exists $z\in\T$  with $z\neq 1$ such that
$z^{k_n}$ converges to $1$ in density.
Then for the probability measure $\mu$ defined by $\mu(\{1\})=\mu(\{z\})=1/2$
\[
\Hat\mu(k_n)=\frac{1+z^{k_n}}{2}
\]
converges to $1$ in density, hence (i) is false.

The last assertion follows immediately from Corollary \ref{cor:coset}.
\end{proof}

Replacing,  in the above proof,  the Ces\`aro limit by the classical limit and convergence in density by classical convergence yields the following.

\begin{thm}\label{thm:extr}
For a sequence $(k_n)$  in $\N$ consider the following assertions:

\begin{iiv}
\item $(k_n)$ is extremal.
\item $(k_n)$ is extremal for discrete measures.
\item $G_\infty((k_n)):=\{z:z^{k_n}\to 1\}=\{1\}$.
\end{iiv}
Then {\upshape (i)}$\Rightarrow${\upshape (ii)}$\Leftrightarrow${\upshape (iii)}. Moreover, {\upshape (i)}$\Leftrightarrow${\upshape (ii)} if $(k_n)$ 
is good.\end{thm}

\begin{remark}[Ergodic sequences]\label{rem:erg}
By the above characterizations or by Proposition \ref{prop:W-seq} (c), every ergodic sequence is Wiener extremal and hence extremal, too.
\end{remark}

 We recall the notation $G_d=:\{\lambda\in\T:\, \lambda^d=1\}$ and observe the following.

\begin{prop}\label{p:bddgapsextr}
Let $(k_n)$ be a subsequence of $\N$  satisfying
\[
\liminf_{n\to \infty}(k_{n+1}-k_n)<\infty. 
\]
Then any   probability measure $\mu$ that is extremal along $(k_n)$ is discrete  with $\supp(\mu)\subseteq \lambda_0 G_d$ for some $d\in\N$ and some $\lambda_0\in\T$.
 As a consequence,
 the following assertions are equivalent:
\begin{iiv}
\item $(k_n)$ is  extremal.
\item $(k_n)$ is  extremal for discrete measures.
\item For every $q\in\N$, $q\geq 2$ there are infinitely many $n$ with $k_n\notin q\N$.
\end{iiv}

\end{prop}

 \begin{remark}\label{rem:unityextr} Note that assertion (iii) above just means that $(k_n)$ is  extremal for roots of unity, i.e., 
\[
\lim_{n\to\infty} \lambda^{k_n}=1,\ \lambda\in\T \text{ root of unity} \quad \Longrightarrow \quad\lambda=1.
\]
\end{remark}

 \begin{remark} A sequence $(k_n)$ with $\liminf_{n\to \infty}(k_{n+1}-k_n)<\infty$ need not be good. An example is given by the sequence $2,4,\dots,2^n,\dots$ where along a subsequence of density $0$ we insert $2^k+1$ right after $2^k$; or see Remark \ref{rem:norel1} for a not good sequence with positive density. Conversely, a good sequence (even if Wiener extremal) need not have such small gaps: Again $k_n=n^2$ is an example. Also, small gaps in $(k_n)$ do not imply that $(k_n)$ would be extremal, an example is $k_n=p_n+1$, $p_n$ the $n^\text{th}$ prime. See Section \ref{sec:pol-primes} for more information.
\end{remark}

\begin{proof}[Proof of Proposition \ref{p:bddgapsextr}]
By assumption there exists an integer 
$d\in \N$ and a subsequence $(n_\ell)_{\ell\in \N}$, such that 
\begin{equation}\label{eq:k_n}
k_{n_\ell+1}-k_{n_\ell}=d \quad \text{for all } \ell\in \N .
\end{equation}
 Let $\mu$ be extremal along $(k_n)$, and let $\theta_n\in[0,2\pi)$ be such that $\hat\mu(k_n)=\ee^{i\theta_n}|\hat 
\mu(k_n)|$. Then 
\[
\int_{[0,2\pi)}(1-\cos (k_nt-\theta_n))\dd\mu(t) 
=1-|\hat \mu(k_n)|
\underset{n\to \infty}\longrightarrow 0.
\]
Hence, 
$(\cos (k_{n_\ell}\cdot -\theta_{n_\ell}))_{\ell \in \N}$ admits a subsequence 
converging $\mu$-a.e. to 1. For simplicity, let us assume that the sequence itself converges $\mu$-a.e. to 1 and that $\theta_{n_\ell}
\to \alpha_0\in [0,2\pi]$ as $\ell\to \infty$. Similarly, we may assume that 
$(\cos (k_{n_\ell+1}\cdot -\theta_{n_\ell+1}))_{\ell \in \N}$ converges $\mu$-a.e 
to $1$ and that $\theta_{n_\ell+1}\to\alpha_1\in [0,2\pi]$ as $\ell\to\infty$.

\medskip\noindent By using \eqref{eq:k_n}, we infer that 
%for some constant $c$ and 
for $\mu$-a.e. $t\in [0,2\pi)$

\[
d t-\alpha_1+\alpha_0 =0\qquad \mod 2\pi .
\]
Hence, $\mu$ is a discrete measure with $\supp(\mu)\subseteq \lambda_0 G_d$ for $\lambda_0=\ee^{i(\alpha_1-\alpha_0)/d} $ and the first assertion is proven.

\medskip 
By Theorem \ref{thm:extr}, it remains to show (iii)$\Rightarrow$(i). Let $\mu$ be extremal along $(k_n)$. By the above we have $\mu=\sum_{j=1}^{d} c_j \delta_{\lambda_0\lambda_j}$, where $c_1,\ldots,c_d\geq 0$ with $\sum_{j=1}^d c_j=1$, $\lambda_0\in\T$ and $\lambda_1,\dots, \lambda_d$ being the $d^\text{th}$ roots of unity. The extremality of $\mu$ implies
\[
|\hat{\mu}(k_n)|^2=\Bigl|\sum_{j=1}^{d} c_j \lambda_j^{k_n}\Bigr|^2=\sum_{j,m=1}^d c_j c_m (\lambda_j\ol{\lambda_m})^{k_n}\underset{n\to \infty}\longrightarrow 1.
\]
By convexity reasons this is possible only if $\lim_{n\to \infty}(\lambda_j\ol{\lambda_m})^{k_n}=1$ whenever $c_jc_m\neq 0$. Thus (iii) and Remark \ref{rem:unityextr} imply $\lambda_j=\lambda_m$ whenever $c_jc_m\neq 0$, meaning that $\mu$ is Dirac.
\end{proof}

\begin{remark}\label{rem:pos-upper-dens-subseq}
For  a subsequence $(k_n)$ and a subset $J\subseteq \N$ of density $1$,  $(k_n)_{n\in J}$ has the same upper density as $(k_n)_{n\in\N}$ by
\[
 \frac{1}{N}\sum_{k_n\leq N,n\notin J} 1\leq \frac{1}{N}\sum_{n
\leq N,n\notin J} 1 \to 0\text{ as }N\to\infty.
 \]
\end{remark}

\begin{lemma}\label{l:densgap}
Let $(k_n)$ be a  subsequence of $\N$ with positive upper density. 
Then $\liminf_{n\to \infty}(k_{n+1}-k_n) <\infty$.
\end{lemma}

\begin{proof}
Assume that $k_{n+1}-k_n\to\infty$ as $n\to\infty$. 
Let $A>0$. There exists $M>0$ such that for every $n\geq M$, 
$k_{n+1}-k_n \geq A$. Hence, for every $n\geq M$ we have 
$k_n\geq k_M+(n-M)A\geq (n-M)A$. Hence, for every $N\in \N$ large enough, 
\[
|\{n:\, k_n\leq N\}|\leq N/A+M\]
 and thus $\ol{d}(k_n)\leq 1/A\to 0$ as $A\to\infty$, resulting in a contradiction.
 \end{proof}

 \begin{remark}
  It is not difficult to exhibit sequences $(k_n)$ with density 0 such that $\liminf_{n\to \infty}(k_{n+1}-k_n) <\infty$. An important example is the sequence of primes $(p_n)$. It is a recent, highly non-trivial result of Zhang that $\liminf_{n\to \infty}(p_{n+1}-p_n) <\infty$, see \cite{Zhang} or the paper \cite{Polymath} by the Polymath project.
 \end{remark}

We have the following characterization of Wiener extremality for sequences with positive upper density. Note that extremality of such sequences was characterized in Proposition \ref{p:bddgapsextr}.

\begin{prop}[Wiener extremality of sequences with positive upper density]\label{prop:pos-upper-dens}
For a subsequence $(k_n)$ with positive upper density 
the following assertions are equivalent:
\begin{iiv}
\item  $(k_n)$ is Wiener extremal.
\item  $(k_n)$ is Wiener extremal for discrete measures.
\item 
 $\ol{d}(\{n:\, k_n\notin q\N\})>0$  for every $q\in\N$, $q\geq 2$.
\end{iiv}
\end{prop}
\noindent Note that assertion (iii) above just means that $(k_n)$ is  Wiener extremal for roots of unity, i.e., 
\[
\Dlim_{n\to\infty} \lambda^{k_n}=1,\ \lambda\in\T \text{ root of unity} \quad \Longrightarrow \quad\lambda=1.
\]
\begin{proof}
It suffices to show the implications (ii)$\Rightarrow$(i) and 
(ii)$\Leftrightarrow$(iii). %

\medskip\noindent (ii)$\Rightarrow$(i): Suppose that $(k_n)$ is Wiener extremal for discrete measures and let $\mu$ be a Wiener extremal measure along $(k_n)$ with decomposition $\mu=\mu_d+\mu_c$ into discrete and continuous parts. By Lemma \ref{lem:KvN} (b) and Remark \ref{rem:pos-upper-dens-subseq} there exists a subsequence $(k_n')$ of $(k_n)$ of positive upper density such that 
$\lim_{n\to\infty}|\hat{\mu}(k_n')|=1$. By Theorem \ref{thm:Wiener}, Lemma \ref{lem:KvN} (a)  there is a subsequence $(m_n)$ of $\N$ of density one with $\lim_{n\to\infty}|\hat{\mu_c}(m_n)|=0$. Denoting by $(\ell_n)$ the non-trivial intersection of $(m_n)$ with $(k_n')$ we obtain
\[
1=\lim_{n\to\infty}|\hat{\mu}(\ell_n)|=\lim_{n\to\infty}|\hat{\mu_d}(\ell_n)|.
\]
Thus $\mu_d$ is also a probability measure, and therefore $\mu=\mu_d$ is Dirac by the assumption. 

\medskip\noindent (ii)$\Leftrightarrow$(iii):
By the equidistribution of $(z^n)$ for any irrational $z$ (i.e., $z$ not a root of unity) combined with Lemma \ref{lem:KvN} (b) and Remark \ref{rem:pos-upper-dens-subseq}
 we obtain the implication:
\[
\Dlim_{n\to\infty}z^{k_n}= 1 \quad\Longrightarrow\quad z\text{  is rational}.
\]
In particular, by Theorem \ref{thm:extr-W} $(k_n)$ is not extremal for discrete measures if and only if there exists $q\in\N$, $q\geq 2$ such that the set of $n$ with $k_n\in q\N$ has density one.
\end{proof}

Thus, the question of characterizing extremality becomes
interesting for sequences of density zero, see Section \ref{sec:pol-primes}.

\begin{ex}[Return time sequences]
Let $(X,\mu,T)$ be an ergodic measure preserving probability system, and let $T$ also denote the corresponding Koopman operator on $\Ell^2(X,\mu)$ defined by $Tf=f\circ T$. Let $A\subseteq X$ with $\mu(A)>0$. We show that 
for almost every $x\in X$ the \emph{return times} sequence $(k_n)$ corresponding to $\{n\in\N:\, T^nx\in A\}$  is Wiener extremal (and hence extremal) whenever $T$ is totally ergodic. Note that return times sequences play an important role for ergodic theorems, see Bourgain's celebrated return times theorem in Bourgain, Furstenberg, Katznelson, Ornstein \cite{B-RTT} and a survey by Assani, Presser \cite{AP}.  Let $\pi_{A,x}(n):=|\{k\leq n:\, T^kx\in A\}|$. We have for $\lambda\in\T$ 
\[
\frac1{\pi_{A,x}(n)}\sum_{k\leq n, T^kx\in A}\lambda^k=\frac{1}{\pi_{A,x}(n)}\sum_{k=1}^n \mathbf{1}_{T^{-k}A}(x) \lambda^k
=\frac{n}{\pi_{A,x}(n)}\frac1n \sum_{k=1}^n( T^k\mathbf{1}_{A})(x) \lambda^k.
\]
 Birkhoff's ergodic theorem and the ergodicity assumption imply that for almost every $x\in X$
\[
\lim_{n\to\infty}\frac{\pi_{A,x}(n)}n=\lim_{n\to\infty}\frac1n \sum_{k=1}^n (T^k\mathbf{1}_{A})(x) =\mu(A), 
\]
i.e., the density of $(k_n)$ equals $\mu(A)$.  Hence, by the 
Wiener--Wintner theorem, see \cite{WW}, for almost every $x$
\[
c(\lambda)=\frac1{\mu(A)}\lim_{n\to\infty}\frac1n \sum_{k=1}^n (T^k\mathbf{1}_{A})(x) \lambda^k =\frac1{\mu(A)}(P_{\ol{\lambda}}\mathbf{1}_{A})(x)\quad \text{for all }\lambda\in\T,
\] 
where $P_{\ol{\lambda}}$ denotes the orthogonal projection onto $\ker(\ol\lambda-T)$. Thus for almost every $x$ the spectrum of the return times sequence is at most countable.
We suppose now that $T$ is totally ergodic and show that $(k_n)$ is Wiener extremal. As in the proof of Proposition  \ref{prop:pos-upper-dens}, if $\limaveN \lambda^{k_n}=1$, then $\lambda$ is rational (i.e., a root of unity). But then total ergodicity implies that $c(\lambda)=0$ for $\lambda\neq 1$, implying $\lambda=1$, and this shows that $(k_n)$ is Wiener extremal.  Note that here total ergodicity cannot be replaced by ergodicity. Indeed, the rotation on two points is ergodic, but for $A$ consisting of one point the return times sequence $(k_n)=2\N$ is not extremal. 
\end{ex}

\begin{ex}[Return time sequences along polynomials]\label{page:polret}
Let $(X,\mu,T)$ be an invertible totally ergodic system, let $T$ denote also its Koopman operator on $\Ell^2(X,\mu)$ and let $\mu(A)>0$. Take a polynomial $P\in\Z[\cdot]$ with $\deg (P)\geq 2$. We show that the return times sequence $(k_n)$ along $P$ corresponding to $\{n\in\N:\, T^{P(n)}x\in A\}$ is ergodic and hence Wiener extremal and extremal for almost every $x$. (That the sequence is Wiener extremal is also for true for linear polynomials, which can be easily deduced from the previous example.)

\medskip\noindent
We let  $\pi_{A,x,P}(n):=|\{k\leq n:\, T^{P(k)}x\in A\}|$ and compute for $\lambda\in\T$
\begin{equation}\label{eq:return-pol-density}
\lim_{n\to\infty}\frac{\pi_{A,x,P}(n)}n=\lim_{n\to\infty}\frac1n \sum_{k=1}^n (T^{P(k)}\mathbf{1}_{A})(x) =\mu(A)\quad \text{a.e. $x\in X$}, 
\end{equation}
where the last equality follows from a.e.~convergence of polynomial averages by Bourgain \cite{B86,B88,B}, from the fact that the rational spectrum factor is characteristic for polynomial averages (see e.g.~Einsiedler, Ward \cite[Sec.{} 7.4]{EW}) and from total ergodicity.

It is a further result of Bourgain that the limit 
\begin{equation}\label{eq:lambdabourgain}
\lim_{n\to\infty}\frac1n \sum_{k=1}^n( T^{P(k)}\mathbf{1}_{A})(x) \lambda^k
\end{equation}
exists for each $\lambda\in \T$ for a.e. $x\in X$, see \cite{EK}.
%\todo{Find the exact ref.} 
Since $\deg( P)\geq 2$, by the spectral theorem,  by the equidistribution of polynomials with at least one irrational non-constant coefficient and  by total ergodicity, the limit in \eqref{eq:lambdabourgain} for almost every $x\in X$
 equals $\mu(A)$ for $\lambda=1$ and $0$ if $\lambda\neq 1$. Combining this with \eqref{eq:return-pol-density} gives 
\begin{align*}
\lim_{n\to\infty}\frac1{\pi_{A,x,P}(n)}\sum_{k\leq n, T^{P(k)}x\in A}\lambda^k
&=\lim_{n\to\infty}\frac{n}{\pi_{A,x,P}(n)}\frac1n \sum_{k=1}^n( T^{P(k)}\mathbf{1}_{A})(x) \lambda^k\\
&=\begin{cases}1&\text{if $\lambda=1$},\\0&\text{otherwise},\end{cases}
\end{align*}
for almost all $x\in X$, meaning that $(k_n)$ is ergodic for almost all $x\in X$.
\end{ex}

\begin{ex}[Double return times sequences]
Let $(X,\mu,T)$ be a {standard} weakly mixing system and let $A,B\subseteq X$ be with $\mu(A),\mu(B)>0$. 
We show that the double return times sequence $(k_n)$ corresponding to $\{n\in\N:\, T^{n}x\in A,\ T^{2n}x\in B\}$ is for almost every $x$ ergodic and hence Wiener extremal and extremal.

\medskip\noindent 
By Bourgain \cite{B-double} the limit 
\[
\lim_{n\to\infty}
\frac1n \sum_{k=1}^n( T^{k}\mathbf{1}_{A})(x)(T^{2k}\mathbf{1}_{B})(x)
\]
exists almost everywhere. Moreover, for weakly mixing systems 
the above limit equals $\mu(A)\mu(B)$ a.e., see, e.g., \cite[Theorem 9.29]{EFHN}. By Assani, Duncan, Moore \cite[Theorem 2.3]{ADM} (or by a product construction), 
for almost every $x$, the limit
\[
\lim_{n\to\infty}\frac1n \sum_{k=1}^n( T^{k}\mathbf{1}_{A})(x)(T^{2k}\mathbf{1}_{B})(x)\lambda^n
\]
exists for each $\lambda\in \T$ and the Host--Kra factor $\mathcal{Z}_2$ is characteristic for such averages (meaning that only the projections of $\mathbf{1}_A$ and $\mathbf{1}_B$ onto this factor contribute to the limit). Since for weakly mixing systems all Host--Kra factors coincide with the fixed factor (see e.g.~Kra \cite[Sect.~6.1,7.3)]{Kra}), the above limit equals $\mu(A)\mu(B)$ for $\lambda=1$ and to zero otherwise. As before, this shows that the double return times sequence is ergodic for almost every $x\in X$
\end{ex}

For more ergodic sequences see Boshernitzan, Kolesnik, Quas, Wierdl \cite{BKQW}.
Note that since the pointwise convergence of weighted averages along primes $\aveN \lambda^n T^{p_n}$ is not yet studied, the return times sequences along primes of the form $\{n:\, T^{p_n}x\in A\}$ are currently out of reach.

\begin{ex}[An extremal sequence which is not Wiener extremal]

Consider the sequence $(k_n)$ defined by the following procedure. Take the sequence $(2n)$ and for $k$ belonging to a fixed subsequence of indices with density zero (e.g., the primes) insert $2k+1$ between $2k$ and $2k+2$.

Clearly, $(k_n)$ is good  with $c(1)=c(-1)=1$ and $c(\lambda)=0$ otherwise. Moreover, for $z\in\T$
\[
\limaveN |z^{k_n}-1|=0 \Longleftrightarrow \limaveN |z^{2n}-1|=0 \Longleftrightarrow z\in\{1,-1\},
\]
whereas $\lim_{n\to\infty}|z^{k_n}-1|=0$ is equivalent to $z=1$. Thus, by Theorems \ref{thm:extr-W} and \ref{thm:extr-W}, $(k_n)$ is extremal but not Wiener extremal. Note that an example of a Wiener extremal measure which is not Dirac is  $\mu$ given by   $\mu(\{1\})=\mu(\{-1\})=1/2$.
\end{ex}

\smallskip

We now go back to Wiener's lemma which in particular implies that a measure $\mu$ is continuous if and only if $\limaveN|\Hat \mu(n)|^2=0$. This motivates the following natural question concerning a characterization of another kind of extremality for subsequences.
\begin{question}
For which subsequences $(k_n)$ of $\N$ and for which 
continuous measures $\mu$ on $\T$ does
\begin{equation}\label{eq:cont}
\limaveN|\Hat \mu(k_n)|^2=0
\end{equation}
hold? For which sequences $(k_n)$ does \eqref{eq:cont} hold for every 
continuous measure? For which sequences $(k_n)$ does \eqref{eq:cont} characterize continuous measures $\mu$?
\end{question}

\begin{remark}
Property \eqref{eq:cont} characterizes continuous measures for ergodic sequences by Proposition \ref{prop:W-seq} (b).
\end{remark}

Note that by Theorem \ref{thm:count}, for sequences which are good with at most countable spectrum, \eqref{eq:cont} holds for all finite continuous measures. The following two examples show however that even for such sequences \eqref{eq:cont} does not characterize continuous measures in general.

\begin{ex}
Consider  $(k_n)$ with $k_n:=2n+1$, which is of course a good sequence with spectrum $\Lambda=\{-1,1\}$ and $c(-1)=-1$. Let
$\mu=\frac12(\delta_{1}+\delta_{-1})$. Then we obtain that
\[
\lim_{N\to\infty}\frac{1}{N}\sum_{n=1}^N|\Hat \mu(2n+1)|^2=0.
\]
\end{ex}

The following observation conjectures a connection between the two kinds of extremality.

\begin{remark}
Consider the following assertions about a sequence $(k_n)$:
\begin{iiv}
\item $(k_n)$ is Wiener extremal for discrete measures and $\frac
1N\sum_{n=1}^N|\Hat\mu(k_n)|^2\to 0$ as $N\to \infty$ for each
continuous measure $\mu$.
\item $(k_n)$  is Wiener extremal.
\item $(k_n)$ is Wiener extremal for discrete measures and $\frac
1N\sum_{n=1}^N|\Hat\mu(k_n)|^2\not\to 1$ as $N\to \infty$ for each
continuous measure $\mu$.
\item $(k_n)$ is Wiener extremal for discrete measures.
\end{iiv}
Then we have the implications (i) $\Rightarrow$ (ii) $\Rightarrow$
(iii) $\Rightarrow$ (iv).
Moreover, for good sequences we have also
(iv) $\Rightarrow$ (ii), i.e., the last three statements are
equivalent.
\begin{proof}
(i) $\Rightarrow$ (ii) follows immediately from the decomposition into the discrete and the continuous part and the triangle inequality (note that by the Koopman--von Neumann Lemma \ref{lem:KvN} we can remove the square in (i)), whereas the implications (ii) $\Rightarrow$(iii) $\Rightarrow$ (iv) are trivial. The last assertion is Theorem \ref{thm:extr-W}.
\end{proof}
\end{remark}

\smallskip
We now discuss connection to rigidity sequences.
Recall that for $\theta\in \T$ a sequence $(k_n)$ is
called a \emph{$\theta$-rigidity sequence} if there is a continuous
probability  measure $\mu$ on $\T$ with $\Hat\mu(k_n)\to \theta$ as
$n\to \infty$. Moreover, $1$-rigidity sequences are called \emph{rigidity sequences}. Note that, although for every $\theta\in\T$, $\theta$-rigid (along some subsequence) continuous measures are typical in the Baire category sense in all probability measures, see Nadkarni \cite{N-book}, to check whether a given sequence $(k_n)$ is rigid or $\theta$-rigid is often a challenge. For more details on such sequences, their properties,  examples and connections to ergodic and operator theory we refer to Nadkarni \cite[Ch.{} 7]{N-book}, Eisner, Grivaux  \cite{EG}, Bergelson, del Junco, Lema\'nczyk, Rosenblatt, \cite{BdJLR}, Aaronson, Hosseini, Lema\'nczyk \cite{AHL}, Grivaux \cite{G13},   Fayad, Kanigowski \cite{FK}, and \cite[Section 4.3]{E-book}.

Theorem \ref{thm:count} (a)  and Corollary \ref{cor:coset} imply in particular a possibly unexpected necessary property of rigidity sequences.
\begin{prop}\label{prop:rigid}
\begin{abc}
\item Suppose the sequence $(k_n)$ is such that there exists a continuous
measure $\mu$ on $\T$ with
\[
\limsup_{N\to\infty}\frac1N\sum_{n=1}^N|\Hat \mu(k_n)|^2>0.
\]
Then either $(k_n)$ is not good, or good with uncountable spectrum.

\item
$\theta$-rigidity sequences are not good.
\end{abc}
\end{prop}
For a consequence for prime numbers, polynomials and polynomials of primes see Proposition \ref{prop:rigid-pol-prim} below.

\begin{remark}
The fact that rigidity sequences cannot be good  with countable spectrum was proved by Bergelson, del Junco, Lema\'nczyk, Rosenblatt \cite[Prop. 2.22]{BdJLR}. Their argument works equally for $\theta$-rigidity 
sequences.
\end{remark}

\begin{ex}
The sequence $(2^n)$ is a rigidity sequence, see Eisner, Grivaux \cite{EG} and  Bergelson, del Junco, Lema\'nczyk, Rosenblatt \cite{BdJLR}, and, as every lacunary sequence,  is not good for the
mean ergodic theorem, see Rosenblatt, Wierdl \cite[Section II.3]{RW}. More examples are sequences satisfying $k_n|k_{n+1}$  or $\lim_{n\to\infty}k_{n+1}/k_n= \infty$, although $\lim_{n\to\infty}k_{n+1}/k_n= 1$ is possible, for details see the two above mentioned papers, \cite{BdJLR} and \cite{EG}.
\end{ex}

 We finally briefly discuss strongly sweeping out sequences. One says that a subsequence $(k_n)_{n\in\N}\subset \N$ has the \emph{strong sweeping out property} if for every aperiodic measure-preserving dynamical system $(X,\mu,T)$ and every $\varepsilon>0$ there exists a set $B\subset X$ with $\mu(B)<\varepsilon$ such that 
 $$
 \liminf_{N\to\infty} \aveN T^{k_n} 1_B=0 \text{ a.e.  and } \limsup_{N\to\infty} \aveN T^{k_n} 1_B=1 \text{ a.e..}
 $$
 In particular, ergodic averages along such sequences diverge a.e.. Strongly sweeping out sequences were discovered by Bellow and have been studied and generalized by many authors, see, e.g.,
 Akcoglu, Bellow, Jones, Losert, Reinhold-Larsson, Wierdl \cite{Akcoglu-etal},
 Losert \cite{Lo},
 Akcoglu, Jones \cite{AJ}, 
 Akcoglu, Jones, Rosenblatt \cite{AJR}, 
 Jones \cite{Jo}, 
 Rosenblatt, Wierdl \cite[Chapter V]{RW}.

In \cite[Theorem 1.1]{Akcoglu-etal}, the authors give a sufficient condition for a sequence to have the strong sweeping out property. By Proposition \ref{prop:c=1}, this condition is never valid for good sequences. Thus all examples presented in \cite{Akcoglu-etal} as a corollary to \cite[Theorem 1.1]{Akcoglu-etal}  are examples of bad sequences, such as lacunary sequences, the sequence $(2^i 3^j)$ (which is non-lacunary by, e.g., \cite{Bosh}) or, for example, sequences of the form $(a_0+a_1d^{n_1}+\ldots a_kd^{n_k})$, where $(a_j)_{j\in\N}\subset \N$, $d\in \N$ and $(n_k)_{k\in\N}\subset \N$ with $\lim_{k\to\infty} n_k=\infty$. (This is in particular another way to see that lacunary sequences are not good.) 

\medskip

On the other hand, the strongly sweeping out sequence produced in Theorem 2.10 of \cite{AJ} is good and even ergodic. In particular, it cannot be a 
$\theta$-rigidity sequence.

\medskip

Let us mention that in the opposite direction, any $\theta$-rigidity sequence 
is strongly sweeping out. 

\begin{prop}
$\theta$-rigidity sequences are strongly sweeping out.
\end{prop}
\begin{proof}
Let $(k_n)$ be a $\theta$-rigidity sequence for some $\theta\in\T$. Then there exists a continuous probability measure $\mu$ on $\T$ such that $\hat\mu(k_n)\underset{n\to \infty}\longrightarrow \theta$. 
Denote 
$$
u_N(\lambda):=\frac1N\sum_{n=1}^N(1-\Re (\ol{\theta}\lambda^{k_n}))
$$
 for all $N\in\N$ and $\lambda\in\T$. The above implies $\lim_{N\to\infty}\int_\T u_N\dd\mu=0$. Since each function $u_N$ is nonnegative, we have convergence of $(u_N)_{N\in\N}$ to $0$ in $\Ell^1(\T,\mu)$. Hence, there exists a subsequence $(N_\ell)_{\ell\in \N}$ of $\N$ such that $u_{N_\ell}\underset {\ell \to \infty}\longrightarrow 0$ $\mu$-a.e.

In particular, 
$$
\frac1{N_\ell}\sum_{n=1}^{N_\ell} 
\lambda^{k_n}
\underset{\ell\to \infty} \longrightarrow \theta\qquad \mbox{for $\mu$-a.e. $\lambda\in\T$}.
$$

Since $\mu$ is continuous, applying Theorem 1.1 of \cite{Akcoglu-etal} (see the remark after it), we infer that for any invertible, bi-measurable and ergodic dynamical system $(X,\Sigma,\nu,T)$ and any $\varepsilon>0$, there exists $A\in \Sigma$ with $\nu(A)<\varepsilon$ and 
$\limsup_{\ell}\frac1{N_\ell}\sum_{n=1}^{N_\ell} \one_A(T^{k_n}(x))=1$ and 
$\liminf_{\ell}\frac1{N_\ell}\sum_{n=1}^{N_\ell} \one_A(T^{k_n}(x))=0$ for $\nu$-a.e. $x$.
\end{proof}

Generally, it is a new research direction to study deeper connections between strongly sweeping out sequences, rigidity sequences, good sequences and Wiener extremal sequences which we will not pursue here further.

%%%%%%%%%%%%%%
%                                       %
% Polynomials & primes  %
%                                       %
%%%%%%%%%%%%%%

\section{Wiener's Lemma along polynomials and primes}\label{sec:pol-primes}

In this section we consider arithmetic sequences such as values of
polynomials, primes and polynomials on primes, inspired by ergodic theorems along such sequences by Bourgain, Wierdl, and Nair, see \cite{B86,B88,B,W,W-diss, N,N1}. Note that all these sequences have density zero (if the degree of the polynomial is greater or equal to two).

The following lemma is classical, see  Vinogradov \cite{V-book}, Hua \cite{H-book}, Rhin \cite{Rh}, and Rosenblatt, Wierdl \cite[Section II.2]{RW}. We present here a quick way to derive it for polynomials of primes from a recent powerful result of Green and Tao \cite[Prop.~10.2]{GT10} on the orthogonality of the modified von Mangoldt function to nilsequences. 

\begin{lemma}\label{lemma:limit-pol-prim}
Let $k_n=P(n)$, $n\in\N$, or $k_n=P(p_n)$, $n\in\N$, where $P$ is an
integer polynomial and $p_n$ denotes the $n^\text{th}$  prime. Then
$c(\lambda)=0$ for every irrational $\lambda\in\T$ (not a root of unity).
\end{lemma}
\begin{proof}[Proof (for polynomials of primes).]
Let
\[
\Lambda'(n):=\begin{cases}
\log n, &\quad \text{if } n\text{ is prime},
\\
0, &\quad \text{otherwise},
\end{cases}
\]
let $\omega\in\N$ and let
\[
W=W_\omega:=\prod_{p \text{ prime},\:p\leq\omega}p.
\]
For $r<W$ coprime to $W$ consider the modified $\Lambda'$-function
\[
\Lambda'_{r,\omega}(n):=\frac{\varphi(W)}{W}\Lambda'(Wn+r), \quad n\in\N,
\]
for the Euler totient function $\varphi$. Let now $P$ be an integer
polynomial and $b_n:=\Lambda'(n)\lambda^{P(n)}$. Since 
$(\lambda^{P(n)})$ can be represented as a Lipschitz nilsequence for a connected, simply connected Lie group, see Green, Tao, Ziegler \cite[Appendix C]{GTZ}, it follows from Green and Tao \cite[Prop.~10.2]{GT10}, see \cite[Lemma 3.2 (b), Cor.~2.2]{E}, that
\begin{align}\label{eq:GT}
\limaveN b_n &= \lim_{\omega\to\infty}\frac{1}{W} \sum_{r<W,\: (r,W)=1} \lim_{N\to\infty} \aveN  b_{Wn+r} \notag\\
&=\lim_{\omega\to\infty}\frac{1}{\varphi(W)} \sum_{r<W,\: (r,W)=1} \lim_{N\to\infty} \aveN  \Lambda'_{r,\omega}(n) \lambda^{P(Wn+r)}\notag\\
&= \lim_{\omega\to\infty}\frac{1}{\varphi(W)} \sum_{r<W,\: (r,W)=1} \lim_{N\to\infty} \aveN  \lambda^{P(Wn+r)}.
\end{align}

Since $P(W \cdot +r)$ for $0<r<W$ are
again integer polynomials and $\lambda$ is irrational (not a root of unity), the sequences
$( \lambda^{P(Wn+r)})_{n\in\N}$ are all equidistributed in $\T$, see
e.g.~Kuipers, Niederreiter \cite[Theorem 1.3.2]{KN}, and hence the right hand side of \eqref{eq:GT} equals zero.
By a classical equality, see e.g.~\cite[Lemma 3.1]{E},
\[
\limaveN b_n =\limaveN \Lambda'(n)\lambda^{P(n)} = \limaveN
\lambda^{P(p_n)}=c(\lambda)
\]
implying the assertion.
\end{proof}

Since for every $P\in\Z[x]$ and every $q\in\N$, the sequence $(P(n))$ is $q$-periodic modulo $q$, the following form of the limit function $c$ can be calculated for the above classes of arithmetic sequences. Here and later we denote
$e(x):=\ee^{2\pi i x}$ for $x\in\R$.
\begin{itemize}
\item For $(P(n))$
\begin{equation}\label{eq:limit-pol}
c(\lambda)=
\begin{cases}\frac{1}{q}\sum\limits_{r=1}^q e(P(r)b/q),\ & \lambda=e(b/q),\:(b,q)=1,\\
0, &\lambda \text{ irrational}.
\end{cases}
\end{equation}
\item
For $(P(p_n))$ the prime number theorem in arithmetic progressions, see e.g.~Davenport \cite{D-book}, implies
\begin{equation}\label{eq:limit-pol-prim}
c(\lambda)=\begin{cases}
\frac{1}{\varphi(q)}\sum\limits_{r\in\{1,\ldots,q\},(r,q)=1} e(P(r)b/q),\ &\lambda=e(b/q),\, (b,q)=1,\\
0, &\lambda \text{ not a root of unity}.
\end{cases}
\end{equation}
\end{itemize}
For more information on the limit in \eqref{eq:limit-pol}  see Kunszenti-Kov\'acs \cite{KK1}, Kunszenti-Kov\'acs, Nittka,  Sauter \cite{KK2}.

An immediate corollary of Lemma \ref{lemma:limit-pol-prim} together with Corollary \ref{cor:coset}
is the following variant of Wiener's lemma.

\begin{thm}[Wiener's lemma along polynomials and
primes]\label{thm:wiener-pol-prim}
Let $k_n=P(n)$ or $k_n=P(p_n)$ for $n\in\N$, where $P$ is a  non-constant
integer polynomial and $p_n$ denotes the $n^\text{th}$ prime. Then
every finite, positive measure $\mu$ on $\T$ satisfies
\begin{align}\label{eq:wiener-pol-prim}
\limaveN |\Hat{\mu}(k_n)|^2&=\sum_{a/q\in\Q} c(e(a/q)) \sum_{\lambda
\text{\atom}} \mu(\{\lambda\}) \mu(\{\lambda e(a/q)\})\\&
\notag
\leq
\sum_{\lambda\in U} \mu(\{\lambda e(\Q)\})^2,
\end{align}
where $U$ is a maximal set of rationally independent atoms of $\mu$.

Moreover, the equality in the above holds if and only if
$c(e(a/q))=1$ holds for every $a,q$ such that
$\lambda_1=e(a/q)\lambda_2$ for some atoms $\lambda_1, \lambda_2$ of
$\mu$.
\end{thm}

\begin{remark}
For example, the equality in the inequality \eqref{eq:wiener-pol-prim}
holds if $\mu$ has no atoms or just one atom $1$, in which case the
limit equals $0$ or $\mu(\{1\})^2$, respectively.
\end{remark}

\begin{thm}\label{thm:pol-primes-extr}
Let $P\in \Z[x]$ be a polynomial.
\begin{abc}
\item  The sequence $(P(n))$ is Wiener extremal
$\Longleftrightarrow$
it is extremal 
$\Longleftrightarrow$
for all $q\geq 2$ there is an $r$ with $P(r)\neq 0$ mod $q$.
\item  The sequence $(p_n)$ of primes is Wiener extremal and hence extremal, too.
\item The sequence $(P(p_n))$ is Wiener extremal 
$\Longleftrightarrow$
it is extremal 
$\Longleftrightarrow$
for all $q\geq 2$ there is an $r$ coprime to $q$ with
$P(r)\neq 0$ mod $q$.
\end{abc}
\end{thm}
\begin{proof}
Assertion (b) and the equivalence of Wiener extremality with the number theoretic characterization in (a) and (c)  follow directly from the formulas for the limit functions \eqref{eq:limit-pol}, \eqref{eq:limit-pol-prim} and Theorem \ref{thm:extr-W}.

Since Wiener extremality implies extremality, it remains to show that in (a) and (c) extemality implies the number theoretic characterization. For (a), assume that there exists $q\in\{2,3,\ldots\}$ such that $P(r)\in q\N$ for all $r\in\N$. Consider the measure $\mu:=(\delta_{e(1/q)}+\delta_{e(-1/q)})/2$. Then 
\[
\hat{\mu}(P(r))=(e(P(r)/q)+e(-P(r)/q))/2=1\quad \text{for every $r$},
\]
so $(P(n))$ is not extremal. The corresponding implication in (c) follows analogously.
\end{proof}

\begin{ex}\label{ex:limit-pol-prim}
We now show that for sequences of the
form $k_n=P(p_n)$, where $p_n$ denotes the $n^\text{th}$ prime, there are
discrete measures $\mu$ which are not Dirac with $\limaveN
|\Hat{\mu}(k_n)|^2=1$. Indeed, let $P$ be some integer polynomial
such that there exists $q\in\N$ with the property that $P(r)=0$ mod $q$ for
every $r\in\N$ coprime to $q$. Then we have in particular
$P(p)/q\in\N$ for every prime $p$ and therefore
\[
c(a/q)=\limaveN e( P(p_n)a/q)=1
\]
for every $a\in\{1,\ldots,q\}$. Thus, the limit of the Wiener
averages along such $(k_n)$ is $1$ for $\mu$ being discrete with $a^q=1$
for every atom $a\in\T$.

 By Theorem \ref{thm:pol-primes-extr} the sequences $(n^2)$ and $(p_n^2)$ are both Wiener extremal.  Note, however, that  $(n^2+5)$ is Wiener extremal, while $(p_n^2+5)$ is trivially not.
More examples are listed below.
\end{ex}

In the following examples we apply the criteria from the previous theorem to various polynomials.
\begin{ex}
Consider the polynomial $P(x)=x^2+a$, $a\in \N$. Then for any $a$ the sequence $(P(n))$ is extremal (for $q\geq 2$ take $r=0$ or $r=1$ in the characterization). We claim that $(P(p_n))$ is extremal if and only if each prime divisor of $a+1$ is greater than $3$. If $2|a+1$, then $P(p_n)$ are all even so that the sequence cannot be extremal. If $3|a+1$, then take $q=3$. $P(1)\equiv a+1\equiv 0$ mod $q$, and $P(2)\equiv a+1\equiv 0$ mod $q$. Hence $(P(p_n))$ cannot be extremal.

\medskip\noindent Suppose the prime divisors of $a+1$ are all different form $2,3$.  Let $q\geq2 $ be arbitrary. If $q\not |a+1=P(1)$, then $r=1$ is a good choice for the criteria in Theorem \ref{thm:pol-primes-extr}. If $q|a+1$, then $P(2)=3+a+1$, so that if $q|P(2)$, then $q|3$ would follow, which is impossible.

\medskip\noindent A concrete examples are $(n^2+4)$ and $(p_n^2+4)$ both being extremal, while $(n^2+2)$ is extremal and $(p_n^2+2)$ is not.
\end{ex}

\begin{ex}
Consider the polynomial $P(x)=x^k+2$. Then for any $k\in \N$, $k\geq 1$ the sequence $(P(n))$ is extremal. We claim that $(P(p_n))$ is extremal if and only if $k$ is odd.
Indeed,  $3|P(2)$ iff $k$ is even. (For $q$ take $r=1$, unless $q=3$ in which case take $r=2$.)
\end{ex}

\begin{ex}
Consider the polynomial $P(x)=ax+b$. The sequence $(an+b)$ is extremal iff $(a,b)=1$. The sequence  $(ap_n+b)$ is extremal iff $a+b\not\equiv 0
$ mod $2$ and $(a,b)=1$. Indeed, that $a,b$ are relatively prime is clearly necessary, while the choice $q=2$ in Theorem \ref{thm:pol-primes-extr} shows that the other condition needs to be fulfilled if   $(ap_n+b)$ is extremal. Conversely, suppose that both conditions hold, and let $q\geq 2$ be arbitrary. If $q\not|a+b$, then we are done with $r=1$. If $q|a+b$, then $q$ is odd, i.e., $(q,2)=1$ and  $P(2)=2a+b\not\equiv 0$ mod $q$, because otherwise $q|a$ and $q|b$, a contradiction.
\end{ex}

Finally, a direct consequence of Proposition \ref{prop:rigid} is the following.
\begin{prop}\label{prop:rigid-pol-prim}
\begin{abc}
\item Let   $P\in \Z[x]$ be as in Theorem \ref{thm:pol-primes-extr} {\upshape (a)}. Then   $(P(n))$ is not a $\theta$-rigidity sequence for any $\theta\in\T$.
\item $(p_n)$ is not a $\theta$-rigidity sequence for any $\theta\in\T$.
\item Let   $P\in \Z[x]$ be as in Theorem \ref{thm:pol-primes-extr} {\upshape (b)}. Then $(P(p_n))$ is not a $\theta$-rigidity sequence for any
$\theta\in \T$.
\end{abc}
\noindent More precisely,  for $(k_n)$ any of the preceding sequences and for
every continuous measure $\mu$ on $\T$ we have
\begin{equation}\label{eq:wiener0}
\limaveN|\Hat \mu(k_n)|^2=0.
\end{equation}
\end{prop}
Note that it has been observed by Eisner, Grivaux \cite{EG} and Bergelson, del Junco, Lema\'nczyk, Rosenblatt \cite{BdJLR} that $(P(n))$ and $(p_n)$ are not
rigidity sequences, but for $(P(p_n))$ this property seems to be new.

\begin{ex}
Consider the sequence $(p_n)$ of primes. We have seen that $(p_n)$ is good with $\Lambda\subseteq e(\Q)$ and $c(-1)=-1$. Let
$\mu=\frac12(\delta_{1}+\delta_{-1})$. Then
we have
\begin{align*}
&\sum_{b/q\in\Q} c(e(b/q))\sum_{\text{$a$ \atom}} \mu(\{e(b/q)a\})\mu(\{a\})\\
&=\mu(\{1\})^2+\mu(\{-1\})^2+2\Re c(-1)\mu(\{1\})\mu(\{-1\})=(1/2-1/2)^2=0.
\end{align*}
So $\mu$ is a discrete measure with
\[
\lim_{N\to\infty}\frac{1}{N}\sum_{n=1}^N|\Hat \mu(p_n)|^2=0.
\]
This shows that for the sequence $(k_n)=(p_n)$ the validity of \eqref{eq:wiener0} for a probability measure $\mu$ does not imply that $\mu$ is continuous.
\end{ex}

%%%%%%%%%%%%%%
%                                       %
%   Operators and sgrs     %
%                                       %
%%%%%%%%%%%%%%

\section{Orbits of operators and operator semigroups}\label{sec:operators}
In this section we study orbits of power bounded linear operators along subsequences and obtain generalizations and extensions of some results of Goldstein and Nagy.

\begin{thm}\label{thm:linop}
Let $(k_n)$ be a good subsequence in $\N$ with at most countable spectrum,
and let $T\in \LLL(H)$ be a linear contraction on the Hilbert
space $H$. Then  for any $x,y\in H$
\begin{align*}
\lim_{N\to\infty}\frac1N\sum_{n=1}^N|{\sprod{T^{k_n}x}{y}}|^2=\sum_{\lambda\in\T}
 c(\lambda)\sum_{a\in \T} \sprod {P_{a\ol{\lambda}}x}{y}\sprod
{P_{a}y}{x},
\end{align*}
where $P_\alpha$ denotes the orthogonal projection onto $\ker(\alpha-T)$.
\end{thm}
Note that by assumption, the summands above are non-zero for at most countably many $\lambda$ and $a$.

\medskip Let $T$ be a linear contraction on a Hilbert space $H$.
We decompose $H$ orthogonally as $H=H_\uni\oplus H_\cnu$ into unitary
and completely non-unitary subspaces both of them being $T$, $T^*$-invariant,
see Sz.-Nagy, Foia{\c{s}} \cite{SzNagyFoiasIV}, or \cite[App.{} D]{EFHN}. The corresponding orthogonal projections are denoted by $P_\uni$ and $P_\cnu$, respectively. By  a result of Foguel \cite{Foguel} the operator $T$ is weakly stable on $H_\cnu$, i.e., $T^n\to 0$ in the weak operator topology (see also \cite[App.{} D]{EFHN} and \cite[Section II.3.2]{E-book}).
\begin{proof}[Proof of Theorem \ref{thm:linop}.]
 We apply the spectral theorem on $H_\uni$ and
obtain the scalar spectral measures $\mu_{x,y}$ on $\T$:
\[
\sprod{T^kx}y=\int_{\T} z^k \dd\mu_{x,y}=\Hat\mu_{x,y}(k)\quad \text{for
all $k\in \Z$, $x,y\in H_\uni$}.
\]
We have
\[
\sprod{T^kx}y=\sprod{T^kP_{\uni}x}{P_{\uni}y}+\sprod{T^k(I-P_{\uni})x}{(I-P_{\uni})y}.
\]
This implies
\begin{align*}
\sum_{n=1}^N|{\sprod{T^{k_n}x}{y}}|^2&=\sum_{n=1}^N|{\sprod{T^{k_n}P_{\uni}x}{P_{\uni}y}}|^2\\
&\quad+\sum_{n=1}^N|{\sprod{T^{k_n}(I-P_{\uni})x}{(I-P_{\uni})y}}|^2\\
&\quad+2\Re\sum_{n=1}^N{\sprod{T^{k_n}P_{\uni}x}{P_{\uni}y}}\sprod{(I-P_{\uni})y} {T^{k_n}(I-P_{\uni})x}.
\end{align*}
Since for any $w\in H_\cnu$ we have $T^kw\to 0$ weakly as $k\to
\infty$, we conclude
\begin{align*}
&\lim_{N\to\infty}\frac1N\sum_{n=1}^N|{\sprod{T^{k_n}x}{y}}|^2=\lim_{N\to\infty}\frac1N\sum_{n=1}^N|{\sprod{T^{k_n}P_{\uni}x}{P_{\uni}y}}|^2\\
&\qquad+\lim_{N\to\infty}\frac1N\sum_{n=1}^N|{\sprod{T^{k_n}(I-P_{\uni})x}{(I-P_{\uni})y}}|^2\\
&\qquad+\lim_{N\to\infty}\frac1N 2\Re\sum_{n=1}^N{\sprod{T^{k_n}P_{\uni}x}{P_{\uni}y}}\sprod{(I-P_{\uni})y}{T^{k_n}(I-P_{\uni})x}\\
&\quad=\lim_{N\to\infty}\frac1N\sum_{n=1}^N|{\sprod{T^{k_n}P_{\uni}x}{P_{\uni}y}}|^2\\
&\quad=\lim_{N\to\infty}\frac1N\sum_{n=1}^N |\Hat\mu_{P_\uni x,P_\uni y}(k_n)|^2\\
&\quad= \sum_{\lambda\in\T}  c(\lambda)\sum_{a \atom}\mu_{P_\uni x,P_\uni y}(\{a\ol{\lambda}\}) \mu_{P_\uni
y,P_\uni x}(\{a\})\\
&\quad=\sum_{\lambda\in\T}  c(\lambda)\sum_{a\in \T} \sprod
{P_{a\ol{\lambda}}x}{y}\sprod {P_{a}y}{x},
\end{align*}
and that was to be proven.
\end{proof}

Next, we need the following easy lemma.
\begin{lemma}\label{lemma:Hilbert-contr}
Let $T$ be a contraction on a Hilbert space $H$, let $x\in H$ and $(k_n)$ be a subsequence in $\N$ such that 
\[
\lim_{n\to\infty}|\sprod{T^{k_n}x}{x}|=\|x\|^2.
\]
Then there is some sequence $(\lambda_n)$ in $\T$ with  $\lim_{n\to\infty}\lambda_n T^{k_n}x=x$. In particular, one has 
\[
\lim_{n\to\infty}|\sprod{T^{k_n}x}{y}|=|\sprod{x}{y}|\quad\text{for all $ y\in H$}.
\] 
\end{lemma}
\begin{proof}
For every $n\in\N$, let $\lambda_n\in\T$ be such that $\lambda_n\sprod{T^{k_n}x}{x}=|\sprod{T^{k_n}x}{x}|$. Since $T$ is a contraction, we have
\[
0\leq \|\lambda_n T^{k_n}x-x\|^2\leq 
2\|x\|^2 - 2\Re(\lambda_n\sprod{T^{k_n}x}{x})=2\|x\|^2 - 2|\sprod{T^{k_n}x}{x}|
\] 
and the right hand side converges to zero by assumption. Thus  
\[
\lim_{n\to\infty}\|\lambda_n T^{k_n}x- x\|=0
\]
holds, and the assertion follows.
\end{proof}

The following proposition characterizes (Wiener) extremal sequences for discrete measures via extremal properties of 
Hilbert space contractions or of general power bounded Banach space operators with relatively compact orbits.

\begin{remark}\label{rem:kronecker}
Let $E$ be a Banach space and $T\in\LLL(E)$ be a power bounded operator on $E$.
Recall that the set of all vectors with relatively compact orbits include eigenvectors and is a closed linear subspace of $E$, see, e.g. Engel, Nagel \cite[Proposition{} A.4]{EN} or \cite[Ch.{} 16]{EFHN}. Thus every vector from the Kronecker (or reversible) subspace $E_r:=\ol{\lin}\{x\in E:\, Tx=\lambda x \text{ for some }\lambda\in\T\}$ has relatively compact orbits.
On the other hand, every $x$ with  relatively compact orbit can be decomposed as $x=x_r+x_s$, where $x_r\in E_r$ and $x_s$ is strongly stable, i.e., satisfies $\lim_{n\to\infty}\|T^n x_s\|=0$, see e.g., \cite[Thm. I.1.16]{E-book}. (This decomposition is usually called the \emph{Jacobs--de Leeuw--Glicksberg decomposition}.)
\end{remark}
\begin{prop}\label{prop:orbits-discr-meas}
Let $(k_n)$ be a subsequence in $\N$.
\begin{abc}
\item 
The following assertions are equivalent:
\begin{iiv}
\item $(k_n)$ is Wiener extremal for discrete measures.

\item For every Hilbert space $H$, every contraction $T\in \LLL(H)$ 
and every $x\in H\setminus\{0\}$  with relatively compact orbit
\begin{equation}\label{eq:relcomp-H}
\limaveN |\sprod{T^{k_n}x}x|^2=\|x\|^4
\end{equation}
implies that $x$ is an eigenvector to a unimodular eigenvalue. 

\item For every power bounded operator $T$ on a Banach space $E$ and every $x\in E\setminus\{0\}$ with relatively compact orbit,
\begin{equation}\label{eq:relcomp-E}
\limaveN |\dprod{T^{k_n}x}{x'}|^2=|\dprod{x}{x'}|^2 \quad\text{for all }x'\in E'
\end{equation}
implies that $x$ is an eigenvector to a unimodular eigenvalue.
\end{iiv}
 \item 
The analogous equivalence holds for extremal sequences  for discrete measures when \eqref{eq:relcomp-H} is replaced by
 \[
\lim_{n\to\infty}|\sprod{T^{k_n}x}x|=\|x\|^2
\]
and \eqref{eq:relcomp-E} is replaced by 
\begin{equation*}
\lim_{n\to\infty}|\dprod{T^{k_n}x}{x'}|=|\dprod{x}{x'}| \quad\text{for all }x'\in E'.
\end{equation*}
 \end{abc}
\end{prop}

\begin{proof} 
(a)  To show (ii) $\Rightarrow$ (i), let $\mu$ be a discrete probability measure with
\[
\limaveN |\hat{\mu}(k_n)|^2=1.
\] 
Consider the Hilbert space $\Ell^2(\T,\mu)$ and the multiplication operator $T$ on it given by $(Tf)(z)=zf(z)$. The assumption on $\mu$ implies
\[
\limaveN |\sprod{T^{k_n}\one}\one|^2=\limaveN |\hat{\mu}(k_n)|^2=1=\|\one\|^4.
\]
Moreover, by $\one=\sum_{a \atom}\one_{\{a\}}$, $T\one_{\{a\}}=a\one_{\{a\}}$ and by Remark \ref{rem:kronecker} $\one$ has relatively compact orbits.

By (ii), there exists $\lambda\in\T$ with $T\one=\lambda\one$, hence $\sum_{a \atom }a\one_{\{a\}}=\sum_{a\atom}\lambda\one_{\{a\}}$. This implies $\one=\one_{\{\lambda\}}$, i.e., $\mu$ is Dirac.

\medskip\noindent (i) $\Rightarrow$ (iii): We may assume without loss of generality that $E$ is spanned by the orbit of $x$ under $T$. By Remark \ref{rem:kronecker},
$T$ has relatively compact orbits on $E$.

Let $\lambda\in \T$. Since $\ol{\lambda}T$ has relatively compact orbits as well, the mean ergodic projections 
\[
P_\lambda := \limaveN \ol{\lambda}^nT^n
\]
onto $\ker(\lambda-T)$ 
exist, where the limit is understood strongly, see, e.g., \cite[Section I.2.1]{E-book} or \cite[Ch.{} 16]{EFHN}. Note that these projections are bounded and commute with $T$.

By Remark \ref{rem:kronecker} (see \cite[Theorem I.1.20]{E-book}, or \cite[Ch.{} 16]{EFHN} for details) $E$ can be decomposed into the reversible part $E_r$ and the stable part $E_s$, both being $T$-invariant, where
\begin{align*}
E_r&=\ol{\lin}\{y\in E:\, Ty=\lambda y \text{ for some }\lambda\in\T\},\\
E_s&=\{y\in E:\,\lim_{n\to\infty}\|T^ny\|=0\}.
\end{align*}
Denote by $P_r$ and $P_s$ the corresponding projections (which commute with $T$) and decompose $x=x_r+x_s$ with $x_r\in E_r$ and $x_s\in E_s$. We have for every $x'\in E'$
\begin{align*}
\aveN |\dprod{ T^{k_n}x_s}{x'}|^2&=
\aveN |\dprod {P_sT^{k_n}x} {x'}|^2=\aveN |\dprod{ T^{k_n}x}{P'_sx'}|^2\\
&\to|\dprod{ x}{P'_sx'}|^2=|\dprod{ x_s}{x'}|^2
\end{align*}
as $N\to\infty$, which together with the strong stability of $T$ on $E_s$ implies $x_s=0$. So that $x=x_r\in E_r$.

Next we argue similarly to the proof of Lemma 6  in Goldstein, Nagy \cite{G3}.  Let $\lambda,\mu\in\T$ be such that $P_\lambda x\neq 0$ and $P_\mu x\neq 0.$ By the Hahn--Banach theorem there exists $z'\in E'$ such that $\dprod{ P_\lambda x}{z'}\neq 0$ and $\dprod {P_\lambda x-P_\mu x}{z'} =0$, i.e., $\dprod{ P_\lambda x}{z'}=\dprod{P_\mu x}{z'}\neq 0$. Observe
\begin{align*}
\aveN |\lambda^{{k_n}}\dprod{ P_\lambda x}{z'}+ \mu^{{k_n}} \dprod { P_\mu x}{z'}|^2
&=\aveN |\dprod{ T^{{k_n}} P_\lambda x+ T^{{k_n}} P_\mu x}{z' }|^2\\
&=\aveN |\dprod { T^{{k_n}} x}{(P'_\lambda+P'_\mu)z }|^2\\
&\to |\dprod{  x}{(P'_\lambda+P'_\mu)z'}|^2=|\dprod{ P_\lambda x}{z'}+\dprod{ P_\mu x}{z' }|^2
\end{align*}
or, equivalently,
\[
\aveN |(\lambda \ol{\mu})^{k_n}+1|^2\to 4
\]
as $N\to\infty$. By Theorem \ref{thm:extr-W} we have $\lambda=\mu$, and therefore $x$ is an eigenvector to a unimodular eigenvalue.

\medskip\noindent (iii) $\Rightarrow$ (ii) follows from Lemma \ref{lemma:Hilbert-contr}.

\medskip\noindent The proof of (b) is, with obvious modifications, the same as for part (a). In particular, one has to apply Theorem \ref{thm:extr} in place of Theorem \ref{thm:extr-W}.
\end{proof}

 For certain sequences including the primes and sequences with positive upper density we have the following extension to power bounded operators on Banach spaces.

\begin{thm}\label{thm:bddgapsextr}
Let $(k_n)$ be  extremal for discrete measures with $\liminf_{n\to \infty}(k_{n+1}-k_n)<\infty$.
Then for every bounded operator $T$ on a Banach space $E$ and every $x\in E\setminus \{0\}$, 
\begin{equation}\label{conv-hyp}
\lim_{n\to\infty} |\dprod{ T^{k_n} x}{x'} |=|\dprod{ x}{x'}| 
\quad \text{ for all }  x'\in E'
\end{equation}
implies that $x$ is an eigenvector with unimodular eigenvalue.
\end{thm}
\begin{proof}
Let $E,T,x$ be as in the statement. By assumption there exists an integer 
$d\in \N$ and a subsequence $(n_\ell)$ such that 
$k_{n_\ell+1}-k_{n_\ell}=d$ for all $\ell\in \N$.
Applying \eqref{conv-hyp} with $(T')^dx'$, we infer that for every $x'\in E'$,
$\lim_{n\to\infty} |\dprod{ T^{k_n+d} x}{x'} |=|\dprod{ T^d x}{x'}|$. 
By the property of $(n_\ell)$ and by the assumption we deduce $|\dprod{ T^dx}{x'} |
=|\dprod{ x}{x'}|$ for every $x'\in E'$. By the Hahn--Banach theorem, there exists $\lambda\in \C$, such that $T^dx=\lambda x$. Moreover, we must have 
$|\lambda | =1$. This implies that $T$ is power bounded on the closed invariant subspace $E_0:={\lin}\{T^jx:j=0,\dots,d-1\}$ and that $x$ has relatively compact orbit. Proposition \ref{prop:orbits-discr-meas} (b) finishes the proof.
\end{proof}

\begin{cor}\label{cor:primesgelfand}
Denote by $p_n$ the $n^\text{th}$ prime.  For every bounded operator
 $T$ on a Banach space $E$ and every $x\in E\setminus \{0\}$, 
\begin{equation*}
\lim_{n\to\infty} |\dprod{ T^{p_n} x}{x'} |=|\dprod{ x}{x'}| 
\quad for\,  every\,  x'\in E'
\end{equation*}
implies that $x$ is an eigenvector with unimodular eigenvalue.
 As a consequence, $\lim_{n\to\infty}T^{p_n}= \Id$ in the weak operator topology implies $T=\Id$.
\end{cor}

\begin{proof}
By Theorem \ref{thm:pol-primes-extr} the sequence $(p_n)$ is extremal, and by the celebrated result of Zhang \cite{Zhang} $(p_n)$ has bounded gaps. So the previous theorem applies.  
\end{proof}

We now characterize Wiener extremal sequences $(k_n)$ via extremal properties of powers $T^{k_n}$ of Hilbert space contractions $T$.

\begin{thm}\label{thm:wienerextrOP}
For a subsequence $(k_n)$ of $\N$ the following are equivalent:
\begin{iiv}
\item $(k_n)$ is Wiener extremal.
\item For every Hilbert space $H$, every linear contraction $T$  
on $H$ and every $x\in H\setminus\{0\}$, if
\begin{equation}\label{eq:Wiener-Hilbert-extr}
\lim_{N\to \infty} \frac1N\sum_{n=1}^N |\sprod{T^{k_n}x}{x}|^2=\|x\|^4,
\end{equation}
then  $x$ is an eigenvector of $T$ with unimodular eigenvalue.
\end{iiv}
\end{thm}
\begin{proof} (ii) $\Rightarrow$ (i): Let $\mu$ be a probability
measure and consider the Hilbert spa\-ce $\Ell^2(\T,\mu)$ and the
unitary operator $T$ defined by $(Tf)(z)=\ol{z} f(z)$. Then $\Hat
\mu(n)=\sprod{T^{n}\one}{\one}$. If $\frac1N\sum_{n=1}^N |\Hat
\mu(k_n)|^2=1=\|\one\|_2^2$, then by the assumption $\one$ is an
eigenfunction of $T$ to some unimodular eigenvalue $\lambda$, meaning
that $\mu=\delta_\lambda$.

\medskip\noindent 
(i) $\Rightarrow$ (ii): We can assume $\|x\|=1$  and decompose
$H=H_\cnu\oplus H_\uni$. We first show that $x\in H_\text{uni}$ which is, with a similar argument, due to Foguel \cite[Theorem 1.3]{Foguel}. Observe
\begin{align*}
|\sprod{T^{n}x}{x}|^2=|\sprod{T^{n}P_\cnu x}{x}|^2+|\sprod{T^{n}P_\uni
x}{x}|^2+2\Re \sprod{T^{n}P_\cnu x}{x} \sprod{ x}{T^{n}P_\uni x}.
\end{align*}
Here the first and the last terms converge to $0$ as $n\to \infty$,
because $T$ is weakly stable on $H_\cnu$. From this we obtain
\[
1=\lim_{N\to \infty} \frac1N\sum_{n=1}^N
|\sprod{T^{k_n}x}{x}|^2=\lim_{N\to \infty} \frac1N\sum_{n=1}^N
|\sprod{T^{k_n}P_\uni x}{P_\uni x}|^2.
\]
In particular $\|P_\uni x\|=\|x\|=1$, implying $x=P_\uni x$. So we can
assume that $H=H_\uni$, and consider the spectral measure $\mu_{x}$
for the unitary operator $T$. Then $\Hat
\mu_x(k_n)=\sprod{T^{k_n} x}{x}$, and $\frac1N\sum_{n=1}^N |\Hat
\mu(k_n)|^2\to 1$.
 By the assumption that $(k_n)$ is Wiener extremal, we obtain that
$\mu_x$ is a Dirac measure, meaning that $x$ is an eigenvector to an
unimodular  eigenvalue.
\end{proof}

\begin{prop}\label{prop:wienerextrOP-WAP}
If $(k_n)$ has positive upper density, 
then the assertions in Theorem \ref{thm:wienerextrOP} are equivalent to 
\begin{iiv}
\item[{\upshape (iii)}]  For every Banach space $E$, every power bounded operator $T\in \LLL(E)$ and every $x\in E\setminus\{0\}$ with relatively weakly compact orbit,  if
\begin{equation}\label{eq:extr-vector}
\lim_{N\to \infty} \frac1N\sum_{n=1}^N |\dprod{T^{k_n}x}{x'}|^2=|\dprod{x}{x'}|^2 \quad \text{for every } x'\in E',
\end{equation}
then  $x$ is an eigenvector of $T$ with unimodular eigenvalue.
\end{iiv}
\end{prop}
\begin{proof}
Indeed, (iii)$\Rightarrow$(ii) follows from Lemma \ref{lemma:Hilbert-contr}. To show the implication (i)$\Rightarrow$(iii) we assume without loss of generality that $T$ is a contraction (by taking an equivalent norm). Let  $(k_n)$ be Wiener extremal and $E$, $T$ and $x$ be as in (iii). We can also assume that $E$ is spanned by the orbit of $x$ under $T$ by restriction to this subspace if necessary. By e.g.~Engel, Nagel \cite[Proposition{} A.4]{EN} or \cite[Sec.{} 16.2]{EFHN}, every orbit of $T$ on $E$ is relatively weakly compact and therefore $T$ admits the Jacobs--de Leeuw--Glicksberg decomposition $E=E_r\oplus E_s$, where
            \begin{align*}
E_r&=\ol{\lin}\{y\in E:\, Ty=\lambda y \text{ for some }\lambda\in\T\},\\
E_s&=\{y\in E:\,\lim_{n\to\infty}T^{m_n}y=0\text{ weakly for some }(m_n) \text{ with density }1\},
\end{align*}
see e.g.~\cite[pp. 9--11]{E-book}  or \cite[Ch.{} 16]{EFHN} for details.
 Let $Q$ be the corresponding projection onto $E_r$, which commutes with $T$. Applying \eqref{eq:extr-vector} to $Q'x'$ instead of $x'$ we see that $Qx$ satisfies  \eqref{eq:extr-vector} as well.  
Since the orbit of $Qx$ is relatively compact and  Wiener extremality implies Wiener extremality for discrete measures, $TQx=\lambda Qx$ for some $\lambda\in\T$ by Proposition \ref{prop:orbits-discr-meas}. It remains to show that $Qx=x$. 

\medskip\noindent  Let $x_s:=x-Qx\in E_s$ and let $(m_n)$ be the sequence of density $1$ with $\lim_{n\to\infty}T^{m_n}x_s=0$ weakly. Applying \eqref{eq:extr-vector} to $(I-Q)'x'$ instead of $x'$ we see that $x_s$ also satisfies  \eqref{eq:extr-vector}.  By Lemma \ref{lem:KvN} (b) and Remark \ref{rem:pos-upper-dens-subseq},  for a fixed $x'\in E'$ with $\|x'\|=1$ and $\dprod{x_s}{x'}=\|x_s\|$ there exists a subsequence $(k_n')$ of $(k_n)$ with positive upper density such that $\lim_{n\to\infty}|\dprod{T^{k_n'}x_s}{x'}|=|\dprod{x_s}{x'}|$.
Let $(\ell_n)$ denote the intersection of $(k_n')$ with $(m_n)$ (which has positive upper density). We have
\[
0=\lim_{n\to\infty}|\dprod{T^{\ell_n}x_s}{x'}|=|\dprod{x_s}{x'}|=\|x_s\|.
\]
The proof is complete.
\end{proof}

Analogously we obtain the following characterization of extremal sequences, 
 where we restate Theorem \ref{thm:bddgapsextr} for the sake of completeness.
\begin{thm}\label{thm:GNgeneral}
For a subsequence $(k_n)$ of $\N$ the following are equivalent:
\begin{iiv}
\item $(k_n)$ is extremal.
\item For every Hilbert space $H$, for every $T$ linear contraction
on $H$, and for every $x\in H\setminus\{0\}$, if
\[
\lim_{n\to \infty} |\sprod{T^{k_n}x}{x}|=\|x\|^2,
\]
then  $x$ is an eigenvector of $T$ with unimodular eigenvalue.
\end{iiv}

Moreover, if $(k_n)$  satisfies $\liminf_{n\to \infty}(k_{n+1}-k_n)<\infty$ (for example, has positive upper density),
then the above are equivalent to:
\begin{iiv}
\item[{\upshape (iii)}]
For every %power 
bounded operator $T$ on a Banach space $E$ and every $x\in E\setminus\{0\}$, 
\[
\lim_{n\to \infty} |\dprod{T^{k_n}x}{x'}|=|\dprod{x}{x'}| \quad \text{for every } x'\in E'
\]
implies that  $x$ is an eigenvector of $T$ with unimodular eigenvalue.
\end{iiv}
\end{thm}
The implication (i)$\Rightarrow$ (iii) of Theorem \ref{thm:GNgeneral} extends
Theorem \ref{abstwienerextr} by Goldstein, Nagy to subsequences.

We finally provide a characterization of identity operators in terms of convergence of ${T^{k_n}}$ to identity in the spirit of Gelfand's $T=\Id$ Theorem, see e.g., \cite[Theorem B.17]{EN}.

\begin{thm} \label{thm:seqTI}
Let $(k_n)$ be a subsequence of $\N$.
\begin{abc}
\item Consider the following assertions:
\begin{iiv}
\item $(k_n)$ is extremal.
\item  Every  Hilbert
space contraction $T$ with
\begin{equation}\label{eq:gelfand}
\lim_{n\to\infty}T^{k_n}= \Id
 \end{equation}
 in the weak operator topology is the identity $\Id$ itself.
\end{iiv}
Then {\upshape (i)}$\Rightarrow${\upshape (ii)}; and {\upshape (i)}$\Leftrightarrow${\upshape(ii)} if $(k_n)$ is good. 
Moreover, one can replace ``Hilbert space contraction $T$'' by ``bounded operator $T$ on a Banach space''
 whenever $(k_n)$ satisfies $\liminf_{n\to \infty}(k_{n+1}-k_n)<\infty$.

\item  The analogous statement holds when in {\upshape (a)} one replaces ``extremal'' by ``Wiener extremal''  in condition  {\upshape(i)}, while in {\upshape (ii)}  \eqref{eq:gelfand} is replaced by 
\[
\Dlim_{n\to\infty}T^{k_n}= \Id.
\]
Moreover, one can replace ``Hilbert space contraction $T$'' by ``Banach space operator $T$ with weakly relatively compact orbits'' whenever $(k_n)$ has positive upper density.
\end{abc}
\end{thm}
\begin{proof}(a) (i)$\Rightarrow$(ii):
By Theorem \ref{thm:GNgeneral}  any $x\in E\setminus \{0\}$ is an eigenvector of $T$ for some eigenvalue $\lambda\in \T$. So that $\lambda^{k_n}\to 1$ as $n\to \infty$, but then by Theorem \ref{thm:extr} $\lambda=1$ follows.

\medskip\noindent (ii)$\Rightarrow$(i): Assume that $(k_n)$  is good. If $(k_n)$ is not extremal, then by Theorem \ref{thm:extr} there is $\lambda\in \T$ with  $\lambda\neq 1$ with $\lambda^{k_n}\to 1$ as $n\to \infty$. The diagonal operator $T:=\diag(1,\lambda)$ on $E:=\C^2$ provides a counterexample to (ii).

\medskip\noindent The proof of (b) is analogous, where one uses Theorem \ref{thm:wienerextrOP} and \ref{thm:extr-W} in place of Theorems \ref{thm:GNgeneral} and \ref{thm:extr}, respectively, as well as Proposition  \ref{prop:wienerextrOP-WAP}.
\end{proof}

For examples of sequences for which the equivalences in Theorems \ref{thm:wienerextrOP},  \ref{thm:GNgeneral} and \ref{thm:seqTI} hold see Remarks \ref{rem:erg}, \ref{rem:pos-upper-dens-subseq} and Theorem \ref{thm:pol-primes-extr}.

The following result shows Wiener extremality of polynomial sequences  for power bounded operators on Hilbert spaces.
\begin{prop}
Let $P\in \Z[\cdot]$ be with $P(\N_0)\subseteq \N_0$ such that  $(P(n))$ is extremal, i.e., satisfies the equivalent conditions in Theorem \ref{thm:pol-primes-extr} {\upshape (a)}. Let further $H$ be a Hilbert space and  $T\in \LLL(H)$ be a power bounded operator on $H$. Then 
\begin{equation*}
\Dlim_{n\to\infty}T^{P(n)}x= x \text{ weakly } \quad \Longrightarrow \quad x\in\Fix( T).
 \end{equation*}
In particular, 
$\Dlim_{n\to\infty}T^{P(n)}= \Id$ in the weak operator topology implies $T=\Id$.
\end{prop}
\begin{proof}
By ter Elst, M\"uller \cite[Theorem 2.4]{tEM}, 
\[
\limaveN T^{P(n)}x=0
\]
whenever $x \perp H_\text{rat}$, where $H_\text{rat}$ is the closed linear span of eigenvectors of $T$ corresponding to rational $\lambda\in \T$. Thus by assumption we have $x\in H_\text{rat}$. Applying the projection $P_{\ker (\lambda - T)}$ to the equality $\Dlim_{n\to\infty}T^{P(n)}x= x$, we see that $P_{\ker (\lambda - T)}x=0$ for every $\lambda\neq 1$ by Theorem \ref{thm:pol-primes-extr}(a), proving $x\in\Fix (T)$.
\end{proof}

\section{The continuous parameter case}\label{sec:cont}

In this section we investigate finite Borel measures on $\R$, the validity of the corresponding version of Wiener's lemma for eventually monotone functions and subsequences, extremal functions and sequences, and some implications for orbits of $C_{0}$-semigroups on Banach spaces. Many proofs go analogously to the case of $\T$ and thus will  be omitted. A substantial difference between the discrete and the continuous parameter case is discussed in Subsection \ref{subsec:cont-arithm} below.

In the following, $f$ : $[0,\ \infty$) $\rightarrow \R$ will be a measurable, locally integrable and eventually monotone function with $\displaystyle \lim_{t\to\infty}f(t)=\pm\infty.$

\subsection{General results and extremal functions}

We call $f$ {\it good} if the limit
\[
c(\theta)\ :=\limaveTau e(\theta f(t))\dd t
\]
exists for every $\theta\in \R$. We further call $f$ {\it ergodic} if it is good with $c=\mathbf{1}_{\{0\}}$. Note that $c$ satisfies $c(0)=1, c(-\theta)=\overline{c(\theta)}$ for all $\theta\in \R$, and is multiplicative on the subgroup $\{\theta:\, |c(\theta)|=1\}$ of $\R$. These assertions can be proved analogously to Remark \ref{rem:c}. For a good function $f$ we call the set $\Lambda:=\{\theta:\, c(\theta)\neq 0\}$ its \emph{spectrum}.

\smallskip

 It is not true in general that for $f$ as above, $\limaveTau e(\theta f(t))\dd t=0$ for
Lebesgue-a.e.~$\theta\in \R$ (take for instance $f=\log $). However,  this is the case if we assume moreover $f$ to be good. In particular, we  have an analogue of Proposition \ref{prop:c=1} for good functions.

\begin{prop}
Let $f$ be good. Then $c$ vanishes almost everywhere with respect to the Lebesgue measure $\leb$.
Moreover,  there exists $\alpha\in \R$ such that $\{\theta\in \R \, : 
\, |c(\theta)|=1\}=\alpha \Z$.
\end{prop}
\begin{proof}
Let $\psi\in \Ell^1(\leb)$. Then, $\theta\mapsto \hat\psi(\theta)=
\int_\R e(\theta t)\psi(t) \dd t\in C_0(\R)$. 
Hence, by dominated convergence,
\[
\int_\R c(t)\psi(t) \dd t=\int_\R\Big(\limaveTau e(tf(\theta))\psi(t)\dd \theta\Big) \dd t=\limaveTau \hat \psi(f(\theta))\dd \theta=0.
\]
Since the latter holds for every $\psi\in \Ell^1(\leb)$, $c$ must vanish 
$\leb$-almost everywhere.

The second statement may be proved as Proposition \ref{prop:c=1}.
\end{proof}

As in the discrete case, we have the following abstract version of Wiener's lemma along functions.

\begin{prop}\label{prop:Wiener-function}(Wiener's lemma along functions, $\R$-version)   Let $f$  be good.
\begin{abc}
\item For every complex Borel measure $\mu$ on $\R$
\[
\limaveTau|\hat{\mu}(f(t))|^{2}\dd t=\int_{\R^{2}}c(\theta-\varphi)\dd(\mu\times\overline{\mu})(\theta,\varphi).
\]
\item The function $f$ is ergodic if and only if
\[
\limaveTau|\hat{\mu}(f(t))|^{2}\dd t=\sum_{a \atom}|\mu(\{a\})|^{2}
\]
holds for every complex Borel measure $\mu$ on $\R.$

\item For an ergodic function $f$ and a Borel probability measure $\mu$ on $\R$ the limit above is 1 if and only if $\mu$  is a Dirac measure.
\end{abc}
\end{prop}
We deduce the following analogue of Corollary \ref{cor:coset}. 

\begin{cor}
Let $f$ be good. 
 For a probability measure $\mu$ on $\R$
 \begin{equation*}
\limaveTau|\hat{\mu}(f(t))|^{2}\dd t=1  
\end{equation*}
holds if and only if $\mu$ is discrete with the property
\begin{equation}\label{eq:dicrete-extr-cont}
c(a-b)=1 \quad \text{for all atoms } a, b.
\end{equation}
In this case,  there exists $\alpha,\beta\in \R$ such that $\mu$ is supported on $\alpha\Z+\beta$.
\end{cor}

For $A\subseteq \R$, denote by $\langle A\rangle$ the additive subgroup generated by $A$. We call $\theta_{1}, \theta_{2}\in \R$  \emph{$A$-{dependent}} if their cosets coincide: $\theta_{1}+\langle A\rangle =\theta_{2}+\langle A
\rangle.$

\begin{thm}
For a good function $f$  with at most countable spectrum $\Lambda$  the following assertions hold:
\begin{abc}
\item
  For every complex Borel measure $\mu$ on $\R$
\[
\limaveTau|\hat{\mu}(f(t))|^{2}\dd t=\sum_{\theta\in\Lambda}c(\theta)\sum_{a\atom}\overline{\mu}(\{a-\theta\})\mu(\{a\}).
\]
In particular, for every continuous complex Borel measure $\mu$  on $\R$
\[
\limaveTau|\hat{\mu}(f(t))|^{2}\dd t=0.
\]

\item  For every Borel probability measure $\mu$  on $\R$
\begin{equation}\label{eq:Wiener-ineq-cont}
 \limaveTau|\hat{\mu}(f(t))|^{2}\dd t\leq\sum_{a\in U}\mu(a+\langle\Lambda\rangle)^{2},  
\end{equation}
where $U$  is a maximal set of $\Lambda$-independent atoms. The equality in {\upshape\eqref{eq:Wiener-ineq-cont}} holds if and only if $\mu$ satisfies {\upshape\eqref{eq:dicrete-extr-cont}} (but is not necessarily discrete).
\end{abc}
\end{thm}

With the continuous notion of density defined by $d(A) := \lim_{t\to\infty}\frac{\leb(A\cap[0,t])}{t}$ ($\leb$ being the Lebesgue measure) and a continuous version of the Koopman--von Neumann lemma, see, e.g., \cite[Lemma III.5.2]{E-book}, one defines extremal and Wiener extremal functions analogously to Definition \ref{def:extr} using the Fourier transform of (finite, Borel) measures on $\R$. Moreover, one has the analogous characterization of such functions as in Section \ref{sec:extr}, where (iii) in Theorem \ref{thm:extr} (and Theorem \ref{thm:extr-W}, respectively) is replaced by:
\begin{equation*}
e(f(t)\theta)\rightarrow 1 \text{  (in density) \quad implies }\theta=0. 
\end{equation*}

%%%   sgrs  %%%

\subsection{Orbits of $C_{0}$-semigroups} 

We now consider orbits of strongly continuous ($C_0$-)semigroups on Banach spaces. For a thorough introduction to the theory of $C_0$-semigroups see, e.g., Engel, Nagel \cite{EN}. 
The following result is a time continuous analogue of Theorem \ref{thm:linop}.
\begin{thm}
 Let $f$ be good with at most countable spectrum, and let $(T(t))_{t \geq 0}$  be a $C_{0}$-semigroup of contractions with generator $A$  on a Hilbert space $H$. Then for any $x, y\in H$
\[
\limaveTau|(T(f(t))x|y)|^{2}\dd t=\sum_{\theta\in \R}c(\theta)\sum_{a\in \R}(P_{a-\theta}x|y)(P_{a}y|x)\ ,
\]
where $P_{\alpha}$ denotes the orthogonal projection onto 
\[
\mathrm{k}\mathrm{e}\mathrm{r}(i\alpha -A)=\bigcap_{t\geq 0}\mathrm{k}\mathrm{e}\mathrm{r}(e(\alpha t)- T(t)).
\]
\end{thm}

Also Lemma \ref{lemma:Hilbert-contr}, Proposition \ref{prop:orbits-discr-meas} and Theorems \ref{thm:wienerextrOP}, \ref{thm:GNgeneral} and \ref{thm:seqTI} transfer to this situation with the natural and obvious changes. In particular, unimodular eigenvalues of the operator $T$ should be replaced by imaginary eigenvalues $a\in i\R$ of the generator $A$ (which correspond to unimodular eigenvalues $\ee^{ta}$ of each $T(t)$). We refer, e.g., to \cite{E-book} for the continuous parameter version of the tools needed, such as the decomposition into a unitary and a completely non-unitar part, the Jacobs--de Leeuw--Glicksberg decomposition for semigroups with relatively compact orbits, etc. 
%{\color{red}Furthermore, we refer to Mustafayev \cite[p.~317]{Mu} for an argument showing the strong continuity of the lifted semigroup  $(\tilde{T}(t))_{t\geq0}$ for the proof of the continuous analogue of (ii)$\Rightarrow$(iii) in Theorem \ref{thm:wienerextrOP}. }

%%%   subseq  %%%

\subsection{Subsequences}

Let now $(k_n)$ be an eventually strictly monotone sequence in $\R$ with
$\lim_{n\to\infty} k_n = \pm\infty$. Consider the function $f:[0,\infty)\to \R$ which is constant $k_n$ on the interval $[n-1,n)$ for $n=1,2,\dots$ and is otherwise $0$. We see that
\[
\frac{1}{N}\int_0^N e(f(t)\theta)\dd t=\aveN e(k_n\theta),\quad
\frac{1}{N}\int_0^N |\hat{\mu}(f(t))|^2 \dd t =
\aveN |\hat{\mu}(k_n)|^2,
\]
thus we are in the setting of subsequences as before, with the difference that the subsequences are
now \emph{real}. We call a subsequence $(k_n)$ in $\R$ \emph{good} if the corresponding $f$ is good, i.e., the
limit
\[
c(\theta) :=\limaveN e(k_n\theta)
\]
exists for every $\theta\in\R$. We further call $(k_n)$ \emph{ergodic} if it is good with 
$c =\mathbf{1}_{\{0\}}$.  For subsequences in $\N$ these notions coincide with the ones in Section \ref{sec:Wiener-subseq}. Moreover, the \emph{spectrum} of $(k_n)$ defined as $\{\theta\in\R:\, c(\theta)\neq 0\}$ coincides with the spectrum of the corresponding function $f$. 

Analogously,
we define $\R$-extremal and $\R$-Wiener extremal sequences using extremality or Wiener extremality of $f$, respectively. Thus $(k_n)$ is $\R$-extremal or $\R$-Wiener extremal if and only if it has the corresponding properties from Definition \ref{def:extr}.

As for functions, the results of Sections \ref{sec:Wiener-subseq}, \ref{sec:extr} and \ref{sec:operators} transfer verbatim up to natural and obvious changes. For example, as in the discrete case, we have the following abstract Wiener lemma along subsequences.

\begin{prop}[Wiener's lemma along subsequences, $\R$-version]\label{prop:wiener-subseq-cont} Let $(k_n)$ be a good sequence in $\R$.
\begin{abc}
\item For every complex Borel measure $\mu$ on $\R$
\[
\limaveN|\hat{\mu}(k_n)|^2=\int_{\R^2}c(t-s)\dd(\mu\times\ol{\mu})(t,s).
\]
\item The sequence $(k_n)$ is ergodic if and only if
\[
\limaveN |\hat{\mu}(k_n)|^2=\sum_{ a\atom} |\mu(\{a\})|^2
\]
holds for every complex Borel measure $\mu$ on $\R$.
\item For an ergodic sequence $(k_n)$ and a Borel probability measure $\mu$ on $\R$ the
limit above is $1$ if and only if $\mu$ is a Dirac measure.
\end{abc}
\end{prop}

%%%   arithm. subseq.  %%%

\subsection{Arithmetic sequences}\label{subsec:cont-arithm}

In this subsection we study examples of $\R$-extremal and $\R$-Wiener extremal sequences. The situation here is very different from the time discrete case. 

We begin with the following observation.
\begin{ex}[Sequences from a copy of $\Z$]
Assume that there is $a>0$ such that $(k_n)\subset a\Z$. Then $(k_n)$ is not $\R$-extremal and hence not $\R$-Wiener extremal even for discrete measures. Indeed, consider $\mu:=\frac12(\delta_{1/a}+\delta_{-1/a})$. Then $\mu$ is not Dirac with
\[
\hat{\mu}(k_n)=\int_\R e(k_n\theta) \dd\mu(\theta)=\frac{e(k_n/a)+e(-k_n/a)}2=1\quad \text{for all $n\in \N$}.
\]
More generally, any measure $\nu$ in $\cconv \{\delta_{k/a}:\, k\in\Z\}$ provides a similar example.
\end{ex}
As a corollary, unlike the discrete case, polynomials with rationally dependent coefficients, primes or such polynomials of primes, though being good with countable spectrum, are \emph{not} $\R$-extremal and hence not $\R$-Wiener extremal. Note that for such sequences $(k_n)$ the periodic  unitary group of translations on $\Ell^2([0,a])$ satisfies $T(an)=I$ for every $n\in\Z$ and thus presents a counterexample to the continuous analogues of (ii) and (iii) in Proposition \ref{prop:orbits-discr-meas}, Theorems \ref{thm:wienerextrOP}, \ref{thm:GNgeneral} and  \ref{thm:seqTI} and 1)--4) from the introduction. 

On the other hand, polynomials with rationally dependent non-constant coefficients and  rationally independent constant term are  $\R$-Wiener extremal and hence $\R$-extremal. Indeed, without loss of generality let $k_n=P(n)+b$ for a polynomial $P$ with coefficients from $a\Z$ and $b$ being rationally independent from $a$. Since the spectrum of $(k_n)$ is countable, it suffices to show that $c(\theta)=1$ implies $\theta=0$ by a continuous analogue of Theorem \ref{thm:extr-W}. As in Section \ref{sec:pol-primes} one can prove that 
\[
c(\theta)=e(b\theta)\limaveN e\left(\tfrac{P(n)}a a\theta\right)=
\begin{cases}
\displaystyle\frac1q \sum_{r=1}^q e(b\theta) e(\tfrac{P(r)d}{aq} ),\quad&\text{if } a\theta=\frac{d}q \in\Q,\\[4ex]
0,\quad &\text{if }a\theta\notin \Q.
\end{cases}
\]
So  $c(\theta)=1$ implies $a\theta=d/q\in\Q$ for some $d\in\Z$, $q\in\N$ and $(b+P(r))d\in a\Z$ for all $r\in\{1,\ldots,q\}$. Since $a,b$ are rationally independent, $d=\theta=0$. Analogously, for such polynomials $P$ the sequence $(P(p_n))$ ($p_n$ denoting the $n^\text{th}$ prime) is $\R$-Wiener extremal and hence $\R$-extremal, too.

Consider finally $P\in\R[\cdot]$ with rationally independent non-constant coefficients. Then for $k_n:=P(n)$, $n\in\N$, we have by Weyl's equidistribution theorem 
\[
c(\theta)=0 \quad\text{for all } \theta\neq 0,
\]
i.e., $(P(n))$ is ergodic. Thus by Proposition \ref{prop:wiener-subseq-cont} (c) $(P(n))$ is $\R$-Wiener extremal and hence $\R$-extremal. Moreover, a suitable modification of Lemma \ref{lemma:limit-pol-prim} using Weyl's equidistribution theorem for polynomials (and the fact that the product of finitely many nilsequences is again a nilsequence) shows that for such polynomials the sequence $(P(p_n))$ is ergodic, and hence $\R$-Wiener extremal.

%%%  sgrs revisited  %%%

\subsection{Orbits of $C_0$-semigroups revisited} 
We thus have the following continuous parameter versions of the results from the introduction being the generalizations of the respective results of Goldstein \cite{G2} and Goldstein, Nagy \cite{G3}. (For the Jacobs--de Leeuw--Glicksberg decomposition for $C_0$-semigroups with relatively compact orbits see, e.g., \cite[Theorem I.1.20]{E-book}.)
\begin{thm}\label{thm:sgr-extr}
Let $(k_n)$ be of the form $(P(n))$ or $(P(p_n))$, where $P\in\R[\cdot]$ has either rationally independent non-constant coefficients, or rationally dependent non-constant coefficients which are rationally independent from the constant coefficient, and we suppose that the leading coefficient of $P$ is positive.
 
\begin{abc}
\item 
Let $(T(t))_{t \geq 0}$  be a $C_{0}$-semigroup of contractions with generator $A$ on a Hilbert space $H$. Then for any $x, y\in H$
\[
\limaveN|(T(k_n)x|y)|^{2}=\sum_{a\in \R}|(P_{a}x|y)|^2,
\]
where $P_a$ denotes the orthogonal projection onto $\ker(a-T)$.
\noindent Moreover, 
\begin{equation*}
\lim_{N\to \infty} \frac1N\sum_{n=1}^N |\sprod{T(k_n)x}{x}|^2=\|x\|^4
\end{equation*}
for $x\neq 0$ implies that  $x$ is an eigenvector of $A$ with imaginary eigenvalue.
\item Let $E$ be a Banach space and $(T(t))_{t \geq 0}$ be a bounded $C_{0}$-semigroup  on $E$ with generator $A$.  Then  
\[
\lim_{N\to \infty} \frac1N\sum_{n=1}^N |\dprod{T(k_n)x}{x'}|^2=|\dprod{x}{x'}|^2 \quad \text{for every } x'\in E'
\]
for $x\in E\setminus\{0\}$ with relatively compact orbit implies that  $x$ is an eigenvector of $A$ with imaginary eigenvalue.
\end{abc}
\end{thm}

\noindent
For example, (a) and (b) in Theorem \ref{thm:sgr-extr} fail for $(\pi n^2)$ or $(p_n)$ but hold for $(n^2+\pi)$ or $(p_n+\sqrt{2})$.

\smallskip 

We finish with the following extremality property of primes being a continuous analogue of Corollary \ref{cor:primesgelfand}, with analogous proof. 

\begin{thm}\label{thm:primesgelfand-R}
Let $a>0$ and $b\in\R$ be rationally independent. Then for every $C_{0}$-semigroup $(T(t))_{t \geq 0}$ on a Banach space $E$ with generator $A$ and every $x\in E\setminus \{0\}$, 
\begin{equation*}
\lim_{n\to\infty} |\dprod{ T(ap_n+b) x}{x'} |=|\dprod{ x}{x'}| 
\quad for\,  every\,  x'\in E'
\end{equation*}
implies that $x$ is an eigenvector of $A$ with imaginary eigenvalue. 
 As a consequence, $\lim_{n\to\infty}T(ap_n+b)= \Id$ in the weak operator topology implies $T(t)=\Id$ for every $t\geq 0$.
\end{thm}

\providecommand{\bysame}{\leavevmode\hbox to3em{\hrulefill}\thinspace}
\providecommand{\MR}{\relax\ifhmode\unskip\space\fi MR }
\providecommand{\MRhref}[2]{%
  \href{http://www.ams.org/mathscinet-getitem?mr=#1}{#2}
}
\providecommand{\href}[2]{#2}

%
 % \bibliographystyle{amsabbrv}
 % \bibliography{wiener-subseq}

\begin{thebibliography}{10}\normalsize

\bibitem{AHL}
J.~Aaronson, M.~Hosseini, and M.~Lema{\'n}czyk, \emph{I{P}-rigidity and
  eigenvalue groups}, Ergodic Theory Dynam. Systems \textbf{34} (2014), no.~4,
  1057--1076.


\bibitem{Akcoglu-etal}
M.~Akcoglu, A.~Bellow, R.~L. Jones, V.~Losert, K.~Reinhold-Larsson, and
  M.~Wierdl, \emph{The strong sweeping out property for lacunary sequences,
  {R}iemann sums, convolution powers, and related matters}, Ergodic Theory
  Dynam. Systems \textbf{16} (1996), no.~2, 207--253.

\bibitem{AJ} M.~Akcoglu, R.~L.~Jones, \emph{Strong sweeping out for block sequences and related ergodic averages}, Proceedings of the Conference on Probability, Ergodic Theory, and Analysis (Evanston, IL, 1997). Illinois J. Math. 43 (1999), no. 3, 447–456.

\bibitem{AJR}
M.~Akcoglu, R.~L.~Jones, J.~M.~Rosenblatt, \emph{The worst sums in ergodic theory}, Michigan Math. J. 47 (2000), no. 2, 265–285. 


\bibitem{ADM}
I.~Assani, D.~Duncan, R.~Moore, \emph{Pointwise characteristic factors for Wiener-Wintner double recurrence theorem}, Ergodic Theory Dynam. Systems \textbf{36} (2016),  no.~4, 1037--1066.

\bibitem{AP}
I.~Assani and K.~Presser, \emph{A survey of the return times theorem}, Ergodic
  theory and dynamical systems, De Gruyter Proc. Math., De Gruyter, Berlin,
  2014, pp.~19--58.

\bibitem{BG}
J.-B. Baillon and S.~Guerre-Delabri{\`e}re, \emph{Optimal properties of
  contraction semigroups in {B}anach spaces}, Semigroup Forum \textbf{50}
  (1995), no.~2, 247--250.

\bibitem{G1}
M.~E. Ballotti and J.~A. Goldstein, \emph{Wiener's theorem and semigroups of
  operators}, Infinite-dimensional systems ({R}etzhof, 1983), Lecture Notes in
  Math., vol. 1076, Springer, Berlin, 1984, pp.~16--22.

\bibitem{BO}
J.~R. Baxter and J.~H. Olsen, \emph{Weighted and subsequential ergodic
  theorems}, Canad. J. Math. \textbf{35} (1983), no.~1, 145--166.

\bibitem{BeL}
A.~Bellow and V.~Losert, \emph{The weighted pointwise ergodic theorem and the
  individual ergodic theorem along subsequences}, Trans. Amer. Math. Soc.
  \textbf{288} (1985), no.~1, 307--345.

\bibitem{Be}
A.~Bellow, \emph{Sur la structure des suites ``mauvaises universelles'' en
  th\'eorie ergodique}, C. R. Acad. Sci. Paris S\'er. I Math. \textbf{294}
  (1982), no.~1, 55--58.

\bibitem{BLRT}
D.~Berend, M.~Lin, J.~Rosenblatt, and A.~Tempelman, \emph{Modulated and
  subsequential ergodic theorems in {H}ilbert and {B}anach spaces}, Ergodic
  Theory Dynam. Systems \textbf{22} (2002), no.~6, 1653--1665.

\bibitem{BdJLR}
V.~Bergelson, A.~del Junco, M.~Lema{\'n}czyk, and J.~Rosenblatt, \emph{Rigidity
  and non-recurrence along sequences}, Ergodic Theory Dynam. Systems
  \textbf{34} (2014), no.~5, 1464--1502.

\bibitem{Bosh}
M.~D.~Boshernitzan, 
\emph{Elementary proof of Furstenberg's Diophantine result},
Proc. Amer. Math. Soc. \textbf{122} (1994), no. 1, 67--70. 

\bibitem{BKQW}
M.~Boshernitzan, G.~Kolesnik, A.~Quas, and M.~Wierdl, \emph{Ergodic averaging
  sequences}, J. Anal. Math. \textbf{95} (2005), 63--103.

\bibitem{B86}
J.~Bourgain, \emph{An approach to pointwise ergodic theorems}, Geometric
  aspects of functional analysis (1986/87), Lecture Notes in Math., vol. 1317,
  Springer, Berlin, 1988, pp.~204--223.

\bibitem{B88}
J.~Bourgain, \emph{On the pointwise ergodic theorem on {$L\sp p$} for
  arithmetic sets}, Israel J. Math. \textbf{61} (1988), no.~1, 73--84.

\bibitem{B}
J.~Bourgain, \emph{Pointwise ergodic theorems for arithmetic sets}, Inst.
  Hautes \'Etudes Sci. Publ. Math. (1989), no.~69, 5--45, With an appendix by
  the author, Harry Furstenberg, Yitzhak Katznelson, and Donald S. Ornstein.

\bibitem{B-double}
J. Bourgain, \emph{Double recurrence and almost sure convergence}, J. Reine Angew. Math. 404 (1990), 140--161.

\bibitem{B-RTT}
J.~Bourgain, H.~Furstenberg, Y.~Katznelson, and D.~S. Ornstein, \emph{Appendix
  on return-time sequences}, Inst. Hautes \'Etudes Sci. Publ. Math. (1989),
  no.~69, 42--45.

\bibitem{D-book}
H.~Davenport, \emph{Multiplicative number theory}, third ed., Graduate Texts in
  Mathematics, vol.~74, Springer-Verlag, New York, 2000, Revised and with a
  preface by Hugh L. Montgomery.


\bibitem{EW} M.~Einsiedler, T.~Ward, \emph{Ergodic Theory: With a View Towards Number Theory}. Graduate Texts in Mathematics, 259, Springer-Verlag, London, 2011.


\bibitem{E}
T.~Eisner, \emph{Nilsystems and ergodic averages along primes}, preprint,
  arXiv:1601.00562, 2016.

\bibitem{E-book}
T.~Eisner, \emph{Stability of operators and operator semigroups}, Operator
  Theory: Advances and Applications, vol. 209, Birkh\"auser Verlag, Basel,
  2010.


\bibitem{EFHN}
T.~Eisner, B.~Farkas, M.~Haase,  R.~Nagel, \emph{Operator Theoretic Aspects of Ergodic Theory}, Graduate Texts in Mathematics, Springer, 2015.


\bibitem{EFNS}
T.~Eisner, B.~Farkas, R.~Nagel, and A.~Ser{\'e}ny, \emph{Weakly and almost
  weakly stable {$C\sb 0$}-semigroups}, Int. J. Dyn. Syst. Differ. Equ.
  \textbf{1} (2007), no.~1, 44--57.

\bibitem{EG}
T.~Eisner and S.~Grivaux, \emph{Hilbertian {J}amison sequences and rigid
  dynamical systems}, J. Funct. Anal. \textbf{261} (2011), no.~7, 2013--2052.

\bibitem{EK}
T.~Eisner, B.~Krause,  \emph{(Uniform) convergence of twisted ergodic averages}, Ergodic Theory Dynam. Systems 36 (2016), no. 7, 2172--2202.

\bibitem{tEM}
A. F. M. ter Elst, V. M\"uller, \emph{A van der Corput-type lemma for power bounded operators}, Math. Z. 285 (2017), no. 1-2, 143--158.


\bibitem{EN}
K.-J. Engel and R.~Nagel, \emph{One-parameter semigroups for linear evolution
  equations}, Graduate Texts in Mathematics, vol. 194, Springer-Verlag, New
  York, 2000, With contributions by S. Brendle, M. Campiti, T. Hahn, G.
  Metafune, G. Nickel, D. Pallara, C. Perazzoli, A. Rhandi, S. Romanelli and R.
  Schnaubelt.

\bibitem{FK}
B.~Fayad and A.~Kanigowski, \emph{Rigidity times for a weakly mixing dynamical
  system which are not rigidity times for any irrational rotation}, Ergodic
  Theory Dynam. Systems \textbf{35} (2015), no.~8, 2529--2534.

\bibitem{Foguel}
S.~R. Foguel, \emph{Powers of a contraction in {H}ilbert space}, Pacific J.
  Math. \textbf{13} (1963), 551--562.

\bibitem{FHK}
N.~Frantzikinakis, B.~Host, and B.~Kra, \emph{Multiple recurrence and
  convergence for sequences related to the prime numbers}, J. Reine Angew.
  Math. \textbf{611} (2007), 131--144.

\bibitem{F-book}
H.~Furstenberg, \emph{Recurrence in ergodic theory and combinatorial number
  theory}, Princeton University Press, Princeton, N.J., 1981, M. B. Porter
  Lectures.

\bibitem{G2}
J.~A. Goldstein, \emph{Extremal properties of contraction semigroups on
  {H}ilbert and {B}anach spaces}, Bull. London Math. Soc. \textbf{25} (1993),
  no.~4, 369--376.

\bibitem{Goldstein1}
J.~A. Goldstein, \emph{Applications of operator semigroups to {F}ourier
  analysis}, Semigroup Forum \textbf{52} (1996), no.~1, 37--47, Dedicated to
  the memory of Alfred Hoblitzelle Clifford (New Orleans, LA, 1994).

\bibitem{G3}
J.~A. Goldstein and B.~Nagy, \emph{An extremal property of contraction
  semigroups in {B}anach spaces}, Illinois J. Math. \textbf{39} (1995), no.~3,
  441--449.

\bibitem{GTZ}
B.~Green, T.~Tao, and T.~Ziegler, \emph{An inverse theorem for the {G}owers
  {$U^{s+1}[N]$}-norm}, Ann. of Math. (2) \textbf{176} (2012), no.~2,
  1231--1372.

\bibitem{GT10}
B.~Green and T.~Tao, \emph{Linear equations in primes}, Ann. of Math. (2)
  \textbf{171} (2010), no.~3, 1753--1850.

\bibitem{G13}
S.~Grivaux, \emph{I{P}-{D}irichlet measures and {IP}-rigid dynamical systems:
  an approach via generalized {R}iesz products}, Studia Math. \textbf{215}
  (2013), no.~3, 237--259.

\bibitem{H}
F.~Hiai, \emph{Weakly mixing properties of semigroups of linear operators},
  Kodai Math. J. \textbf{1} (1978), no.~3, 376--393.

\bibitem{H-book}
L.-K. Hua, \emph{Die {A}bsch\"atzung von {E}xponentialsummen und ihre
  {A}nwendung in der {Z}ahlentheorie}, Enzyklop\"adie der mathematischen
  Wissenschaften: Mit Einschluss ihrer Anwendungen, Bd. I, vol.~2, B. G.
  Teubner Verlagsgesellschaft, Leipzig, 1959.

\bibitem{Jo}
R.~L.~Jones, 
\emph{Strong sweeping out for lacunary sequences}, Chapel Hill Ergodic Theory Workshops, 137–144,
Contemp. Math., 356, Amer. Math. Soc., Providence, RI, 2004. 

\bibitem{Kah}
J.-P. Kahane, \emph{Sur les coefficients de {F}ourier-{B}ohr}, Studia Math.
  \textbf{21} (1961/1962), 103--106.

\bibitem{Koopman-vonNeumann1932}
B.~O. Koopman and J.~von Neumann, \emph{{Dynamical systems of continuous
  spectra}}, Proc. Nat. Acad. Sci. U.S.A. \textbf{18} (1932), 255--263.

\bibitem{Kra}
B.~Kra, \emph{Ergodic methods in additive combinatorics}, Additive Combinatorics, 103, CRM Proc. Lecture Notes, 43, Amer. Math. Soc., Providence, RI, 2007.

\bibitem{K}
B.~Krause, \emph{Polynomial ergodic averages converge rapidly: Variations on a
  theorem of {B}ourgain}, preprint, arXiv:1402.1803v1, 2014.

\bibitem{KN}
L.~Kuipers and H.~Niederreiter, \emph{Uniform distribution of sequences},
  Wiley-Interscience [John Wiley \& Sons], New York-London-Sydney, 1974, Pure
  and Applied Mathematics.

\bibitem{KK1}
D.~Kunszenti-Kov{\'a}cs, \emph{On the limit of square-{C}es\`aro means of
  contractions on {H}ilbert spaces}, Arch. Math. (Basel) \textbf{94} (2010),
  no.~5, 459--466.

\bibitem{KK2}
D.~Kunszenti-Kov{\'a}cs, R.~Nittka, and M.~Sauter, \emph{On the limits of
  {C}es\`aro means of polynomial powers}, Math. Z. \textbf{268} (2011),
  no.~3-4, 771--776.

\bibitem{LOT}
M.~Lin, J.~Olsen, and A.~Tempelman, \emph{On modulated ergodic theorems for
  {D}unford-{S}chwartz operators}, Proceedings of the {C}onference on
  {P}robability, {E}rgodic {T}heory, and {A}nalysis ({E}vanston, {IL}, 1997),
  vol.~43, 1999, pp.~542--567.

\bibitem{L}
P.-K. Lin, \emph{A remark on contraction semigroups on {B}anach spaces}, Bull.
  London Math. Soc. \textbf{27} (1995), no.~2, 169--172.

\bibitem{Lo}
V.~Losert, \emph{A remark on the strong sweeping out property}, Convergence in ergodic theory and probability (Columbus, OH, 1993), 291–294, Ohio State Univ. Math. Res. Inst. Publ., 5, de Gruyter, Berlin, 1996.

\bibitem{M}
M.~Mirek, \emph{{$\ell\sp p(\Bbb{Z})$}-boundedness of discrete maximal
  functions along thin subsets of primes and pointwise ergodic theorems}, Math.
  Z. \textbf{279} (2015), no.~1-2, 27--59.

%\bibitem{Mu}
%H.~S. Mustafayev, \emph{Mixing type theorems for one-parameter semigroups of operators}, Semigroup Forum \textbf{92} (2016), no.~2, 311--334.

\bibitem{N-book}
M.~G. Nadkarni, \emph{Spectral theory of dynamical systems}, Texts and Readings
  in Mathematics, vol.~15, Hindustan Book Agency, New Delhi, 2011, Reprint of
  the 1998 original.

\bibitem{N1}
R.~Nair, \emph{On polynomials in primes and {J}. {B}ourgain's circle method
  approach to ergodic theorems}, Ergodic Theory Dynam. Systems \textbf{11}
  (1991), no.~3, 485--499.

\bibitem{N}
R.~Nair, \emph{On polynomials in primes and {J}. {B}ourgain's circle method
  approach to ergodic theorems. {II}}, Studia Math. \textbf{105} (1993), no.~3,
  207--233.


\bibitem{Polymath} Polymath, D. H. J.
\emph{Variants of the Selberg sieve, and bounded intervals containing many 
primes},
Res. Math. Sci.  1  (2014), Art. 12, 83 pp. 

\bibitem{Rh}
G.~Rhin, \emph{Sur la r\'epartition modulo {$1$} des suites {$f(p)$}}, Acta
  Arith. \textbf{23} (1973), 217--248.

\bibitem{Rosenblatt} J. M. Rosenblatt, \emph{ Norm convergence in ergodic theory and the behavior of Fourier transforms}, Canad. J. Math.  46  (1994),  no. 1, 184-199.


\bibitem{RW}
J.~M. Rosenblatt and M.~Wierdl, \emph{Pointwise ergodic theorems via harmonic
  analysis}, Ergodic theory and its connections with harmonic analysis
  ({A}lexandria, 1993), London Math. Soc. Lecture Note Ser., vol. 205,
  Cambridge Univ. Press, Cambridge, 1995, pp.~3--151.

\bibitem{SzNagyFoiasIV}
B.~Sz.-Nagy and C.~Foia{\c{s}}, \emph{Sur les contractions de l'espace de
  {H}ilbert. {IV}}, Acta Sci. Math. Szeged \textbf{21} (1960), 251--259.

\bibitem{V-book}
I.~M. Vinogradov, \emph{The method of trigonometrical sums in the theory of
  numbers}, Dover Publications, Inc., Mineola, NY, 2004, Translated from the
  Russian, revised and annotated by K. F. Roth and Anne Davenport, Reprint of
  the 1954 translation.

\bibitem{W-diss}
M.~Wierdl, \emph{Almost everywhere convergence and recurrence along the
  primes}, Ph.D. thesis, Ohio State University, 1989.

\bibitem{W}
M.~Wierdl, \emph{Pointwise ergodic theorem along the prime numbers}, Israel J.
  Math. \textbf{64} (1988), no.~3, 315--336 (1989).

\bibitem{WW} N. Wiener and A. Wintner, \emph{Harmonic analysis and ergodic theory}, American Journal of Mathematics, \textbf{63} 415--426 (1941).




\bibitem{WZ}
T.~D. Wooley and T.~D. Ziegler, \emph{Multiple recurrence and convergence along
  the primes}, Amer. J. Math. \textbf{134} (2012), no.~6, 1705--1732.
  
  \bibitem{Zhang} Y. Zhang, \emph{Bounded gaps between primes,} Ann. of Math. (2)  179  (2014),  no. 3, 1121-1174. 

\bibitem{ZK-primes}
P.~Zorin-Kranich, \emph{Variation estimates for averages along primes and
  polynomials}, J. Funct. Anal. \textbf{268} (2015), no.~1, 210--238.

\end{thebibliography}
% \parindent0pt

\end{document}